\documentclass{amsart} 

\title{Multiplicity of characters of finite reductive groups and Drinfeld doubles}
 \author{GyeongHyeon Nam}
 
 \address{G.Nam - Department of Mathematics and Systems Analysis, Aalto University, Espoo 02150, Finland}
\email{\href{mailto:gyeonghyeon.nam@aalto.fi}{gyeonghyeon.nam@aalto.fi}, \href{mailto:gyeonghyeon.alg@gmail.com}{gyeonghyeon.alg@gmail.com}}

\usepackage{fullpage}
\usepackage{amsmath, amsthm, amssymb, amsfonts, amsthm, mathrsfs}
\usepackage{textgreek, makecell, enumerate, bm, mathtools, thmtools, float, cite, wrapfig, multicol}
\usepackage[dvipsnames]{xcolor}
\usepackage{colonequals}
\usepackage{hyperref}
\hypersetup{
linktoc = section,
colorlinks = true,
urlcolor = Blue,
linkcolor = Red,
citecolor = RoyalBlue	
}
\usepackage[nobysame, alphabetic]{amsrefs}
\usepackage{comment}

\theoremstyle{plain}
\newtheorem{thm}{Theorem}
\newtheorem{lem}[thm]{Lemma}
\newtheorem{prop}[thm]{Proposition}

\newtheorem{cor}[thm]{Corollary}
\newtheorem{ques}[thm]{Question}
\theoremstyle{definition}
\newtheorem{defe}[thm]{Definition}  
 
\theoremstyle{remark}
\newtheorem{rem}[thm]{Remark}

\definecolor{red}{rgb}{1,0,0}
\definecolor{orange}{rgb}{1,0.5,0}
\definecolor{purple}{rgb}{.5,.2,.8}
\definecolor{blue}{rgb}{.2,.2,.8}
\definecolor{green}{rgb}{.4,.6,.4}

\newcommand{\gncom}[1]{\noindent  \textcolor{purple}{$[\star$ N: #1 $\star]$}}

\newcommand{\nc}{\newcommand}
\newcommand{\rc}{\renewcommand}
\nc{\on}{\operatorname}
\nc{\Fq}{\mathbb{F}_q}
\nc{\fg}{\mathfrak{g}}
\nc{\ft}{\mathfrak{t}}

\nc{\ra}{\rightarrow}
\newcommand{\fT}{\mathfrak{T}}

\rc{\AA}{\mathbb{A}}	
\nc{\BB}{\mathbb{B}}	
\nc{\CC}{\mathbb{C}}	
\nc{\DD}{\mathbb{D}}	
\nc{\EE}{\mathbb{E}}	
\nc{\FF}{\mathbb{F}}	
\nc{\GG}{\mathbb{G}}	
\nc{\HH}{\mathbb{H}}	
\nc{\II}{\mathbb{I}}	
\nc{\JJ}{\mathbb{J}}	
\nc{\KK}{\mathbb{K}}	
\nc{\LL}{\mathbb{L}}	
\nc{\MM}{\mathbb{M}}	
\nc{\NN}{\mathbb{N}}	
\nc{\OO}{\mathbb{O}}	
\nc{\PP}{\mathbb{P}}	
\nc{\QQ}{\mathbb{Q}}	
\nc{\RR}{\mathbb{R}}	
\rc{\SS}{\mathbb{S}}	
\nc{\TT}{\mathbb{T}}	
\nc{\UU}{\mathbb{U}}	
\nc{\VV}{\mathbb{V}}	
\nc{\WW}{\mathbb{W}}	
\nc{\XX}{\mathbb{X}}	
\nc{\YY}{\mathbb{Y}}	
\nc{\ZZ}{\mathbb{Z}}	

\nc{\bA}{\mathbf{A}}	
\nc{\bB}{\mathbf{B}}	
\nc{\bC}{\mathbf{C}}	
\nc{\bD}{\mathbf{D}}	
\nc{\bE}{\mathbf{E}}	
\nc{\bF}{\mathbf{F}}	
\nc{\bG}{\mathbf{G}}	
\nc{\bH}{\mathbf{H}}	
\nc{\bI}{\mathbf{I}}	
\nc{\bJ}{\mathbf{J}}	
\nc{\bK}{\mathbf{K}}	
\nc{\bL}{\mathbf{L}}	
\nc{\bM}{\mathbf{M}}	
\nc{\bN}{\mathbf{N}}	
\nc{\bO}{\mathbf{O}}	
\nc{\bP}{\mathbf{P}}	
\nc{\bQ}{\mathbf{Q}}	
\nc{\bR}{\mathbf{R}}	
\nc{\bS}{\mathbf{S}}	
\nc{\bT}{\mathbf{T}}	
\nc{\bU}{\mathbf{U}}	
\nc{\bV}{\mathbf{V}}	
\nc{\bW}{\mathbf{W}}	
\nc{\bX}{\mathbf{X}}	
\nc{\bY}{\mathbf{Y}}	
\nc{\bZ}{\mathbf{Z}}	

\nc{\calA}{\mathcal{A}}	
\nc{\calB}{\mathcal{B}}	
\nc{\calC}{\mathcal{C}}	
\nc{\calD}{\mathcal{D}}	
\nc{\calE}{\mathcal{E}}	
\nc{\calF}{\mathcal{F}}	
\nc{\calG}{\mathcal{G}}	
\nc{\calH}{\mathcal{H}}	
\nc{\calI}{\mathcal{I}}	
\nc{\calJ}{\mathcal{J}}	
\nc{\calK}{\mathcal{K}}	
\nc{\calL}{\mathcal{L}}	
\nc{\calM}{\mathcal{M}}	
\nc{\calN}{\mathcal{N}}	
\nc{\calO}{\mathcal{O}}	
\nc{\calP}{\mathcal{P}}	
\nc{\calQ}{\mathcal{Q}}	
\nc{\calR}{\mathcal{R}}	
\nc{\calS}{\mathcal{S}}
\nc{\calT}{\mathcal{T}}	
\nc{\cU}{\mathcal{U}}	
\nc{\calV}{\mathcal{V}}	
\nc{\calW}{\mathcal{W}}
\nc{\calX}{\mathcal{X}}	
\nc{\calY}{\mathcal{Y}}	
\nc{\calZ}{\mathcal{Z}}

\nc{\fraka}{\mathfrak{a}}
\nc{\frakb}{\mathfrak{b}}
\nc{\frakc}{\mathfrak{c}}
\nc{\frakd}{\mathfrak{d}}
\nc{\frake}{\mathfrak{e}}
\nc{\frakf}{\mathfrak{f}}
\nc{\frakg}{\mathfrak{g}}
\nc{\frakh}{\mathfrak{h}}
\nc{\fraki}{\mathfrak{i}}
\nc{\frakj}{\mathfrak{j}}
\nc{\frakk}{\mathfrak{k}}
\nc{\frakl}{\mathfrak{l}}
\nc{\frakm}{\mathfrak{m}}
\nc{\frakn}{\mathfrak{n}}
\nc{\frako}{\mathfrak{o}}
\nc{\frakp}{\mathfrak{p}}
\nc{\frakq}{\mathfrak{q}}
\nc{\frakr}{\mathfrak{r}}
\nc{\fraks}{\mathfrak{s}}
\nc{\frakt}{\mathfrak{t}}
\nc{\fraku}{\mathfrak{u}}
\nc{\frakv}{\mathfrak{v}}
\nc{\frakw}{\mathfrak{w}}
\nc{\frakx}{\mathfrak{x}}
\nc{\fraky}{\mathfrak{y}}
\nc{\frakz}{\mathfrak{z}}

\nc{\frakA}{\mathfrak{A}}
\nc{\frakB}{\mathfrak{B}}
\nc{\frakC}{\mathfrak{C}}
\nc{\frakD}{\mathfrak{D}}
\nc{\frakE}{\mathfrak{E}}
\nc{\frakF}{\mathfrak{F}}
\nc{\frakG}{\mathfrak{G}}
\nc{\frakH}{\mathfrak{H}}
\nc{\frakI}{\mathfrak{I}}
\nc{\frakJ}{\mathfrak{J}}
\nc{\frakK}{\mathfrak{K}}
\nc{\frakL}{\mathfrak{L}}
\nc{\frakM}{\mathfrak{M}}
\nc{\frakN}{\mathfrak{N}}
\nc{\frakO}{\mathfrak{O}}
\nc{\frakP}{\mathfrak{P}}
\nc{\frakQ}{\mathfrak{Q}}
\nc{\frakR}{\mathfrak{R}}
\nc{\frakS}{\mathfrak{S}}
\nc{\frakT}{\mathfrak{T}}
\nc{\frakU}{\mathfrak{U}}
\nc{\frakV}{\mathfrak{V}}
\nc{\frakW}{\mathfrak{W}}
\nc{\frakX}{\mathfrak{X}}
\nc{\frakY}{\mathfrak{Y}}
\nc{\frakZ}{\mathfrak{Z}}

\nc{\Lie}{\on{Lie}}
\nc{\GL}{\on{GL}}
\nc{\PGL}{\on{PGL}}
\nc{\SL}{\on{SL}}
\nc{\Sp}{\on{Sp}}
\nc{\GSp}{\on{GSp}}
\nc{\SO}{\on{SO}}
\nc{\Or}{\on{O}}
\nc{\gl}{\on{\mathfrak{gl}}}
\rc{\sl}{\on{\mathfrak{sl}}}

\nc{\Mat}{\on{Mat}}
\nc{\Fun}{\on{Fun}}
\nc{\Aut}{\on{Aut}}
\nc{\End}{\on{End}}
\nc{\Hom}{\on{Hom}}
\nc{\Sym}{\on{Sym}}
\nc{\Span}{\on{span}}
\nc{\Irr}{\on{Irr}}
\nc{\Uch}{\on{Uch}}
\nc{\Spec}{\on{Spec}}

\nc{\Ind}{\on{Ind}}
\nc{\Res}{\on{Res}}
\nc{\stab}{\on{stab}}
\nc{\orb}{\on{orb}}
\rc{\ker}{\on{ker}}
\nc{\im}{\on{im}}
\nc{\tr}{\on{tr}}
\nc{\ord}{\on{ord}}
\nc{\rank}{\on{rank}}
\nc{\Tor}{\on{Tor}}
\nc{\Ad}{\on{Ad}}

\nc{\Id}{\on{Id}}
\nc{\Log}{\on{Log}}
\nc{\Exp}{\on{Exp}}
\nc{\Frac}{\on{Frac}}
\nc{\diag}{\on{diag}}
\nc{\D}{\on{D}}

\nc{\St}{\mathrm{St}}
\nc{\triv}{\mathrm{triv}}
\nc{\sgn}{\mathrm{sgn}}
\nc{\reg}{\mathrm{reg}}

\nc{\op}{\mathrm{op}}
\nc{\ad}{\mathrm{ad}}
\rc{\ss}{\mathrm{ss}}
\nc{\HLV}{\mathrm{HLV}}
\nc{\GIT}{\mathrm{GIT}}
\nc{\stack}{\mathrm{stack}}

\allowdisplaybreaks

\begin{document}


\maketitle

\begin{abstract}
In this paper, we compute the multiplicities of tensor products of  unipotent almost characters and Deligne–Lusztig characters of a finite reductive group $G^F$, and these multiplicities are related to the ring structure of the complex irreducible characters of $G^F$. In addition, we consider Frobenius--Schur indicators of modules over the Drinfeld doubles of finite reductive groups. In the final section, we study the multiplicities of tensor products of  unipotent almost characters.
\end{abstract}

{
  \hypersetup{linkcolor=black}
  \tableofcontents
}
 \section{Introduction}
 
 Let $G$ be an untwisted connected split reductive group over an algebraically closed field $k$ (with positive characteristic) with the untwisted Frobenius map $F$. 
 When we  consider (complex) irreducible characters of $G^F$ (more generally, finite groups), one fundamental question is about the ring structure of $\widehat{G^F}$, where $\widehat{H}$ is the set of irreducible characters of a finite group $H$.
 This space can be considered as the space spanned by isomorphism classes of finite dimensional complex irreducible representations, and the operations are the direct sum and the tensor product.
With this view-point, our main problem is to compute the multiplicity
 \[
 \left\langle \chi_1\otimes \cdots \otimes \chi_m,\overline{\chi}\right\rangle=\left\langle \chi_1\otimes \cdots \otimes \chi_m\otimes {\chi},1\right\rangle
 \]
 for $\chi_1, \ldots , \chi_m,\chi\in \widehat{G^F}, $ and this gives the decomposition $\chi_1\otimes \cdots \otimes \chi_m = \sum_{\chi \in \widehat{G^F}} \left\langle \chi_1\otimes \cdots \otimes \chi_m\otimes {\chi},1\right\rangle \overline{\chi}$.
 
 We consider this problem over Deligne-Lusztig characters and  unipotent almost characters in this paper. One motivation for studying these tensor products is that some of them can be regarded as a multiplicity variant of the work of \cite{KNP}. In \cite{KNP}, authors considered character varieties over regular unipotent and regular semisimple conjugacy classes, and the main point of the computation is to sum over semisimple characters in induced representations over split semisimple elements in $\check{G}^F$ (here, split semisimple element means that it is contained in a split maximal torus). Recall that the product of the Steinberg character and a Deligne–Lusztig character has vanishing character values except on particular semisimple elements. This provides a situation analogous to the computation in \cite{KNP}. Another intuition will be given in \S\ref{sss:121}.

In addition, we calculate a related term to the size of $G_m(g,z):= \{x\in G^F\,|\, x^m=(gx)^m=z\}$, which is an important term to consider the Frobenius-Schur indicators of the modules over the Drinfeld double $D(G^F)$. We also consider the multiplicities of tensor products involving only  unipotent almost characters, drawing on ideas from Letellier’s work; cf. \cite{letellier2013tensor}.  

 \subsection{Notation} In this paper, $G$ is an  untwisted  split connected reductive group over $k $ and the untwisted Frobenius map $F$. We denote its centre by $Z(G)$ and the set of semisimple (resp. unipotent) elements in $G$ by $G^{ss}$ (resp. $G^{uni}$).
Moreover, we assume that the derived subgroup $[G, G]$ of $G$ is simply connected, and the size $|k^F|=q$ is large enough so that every maximal torus of $G^F$ is non-degenerate, cf. \cite[Proposition 3.6.6]{carter1985finite}.

For each $g\in G$, we have $g=g_sg_u=g_ug_s$ for a semisimple element $g_s$ and a unipotent element $g_u$ with $g_u \in G_{g_s}$ under the Jordan decomposition. Note that if $g\in G^F$, then $g_s,g_u\in G^F$ since the Jordan decomposition is unique.  Let us denote its centraliser subgroup $C_G(g)$ by $G_g$ for $g\in G$, and this is called a  pseudo-Levi subgroup when $g$ is semisimple, cf. \cite[\S2.2]{KNWG}. (Note that every pseudo-Levi subgroup is connected from the assumption that $[G,G]$ is simply connected.) Let $G_{\mathrm{pls}}^F$ denote the set of finite pseudo-Levi subgroups $G_s^F$ of semisimple elements $s$ in $G^F$. For a pseudo-Levi subgroup $\fT$ of $G$ with its maximal torus $T$, $Q_T^\fT$ is the Green function, cf. \cite[Definition 4.1]{DL76} or \cite[Definition 2.2.15]{geck2020character}. For a set $X$ with $G^F$-action, we denote the $G^F$-orbit of $x$ by $[x]$.
 
Let $W(T)$ be the Weyl group of $G$ over a maximal torus $T$, i.e., $N_G(T)/T$ (we drop $G$ when it is obvious) and $R_{T_w}^G(1)$ the Deligne-Lusztig character over a $w$-twisted torus $T_w$ for a split maximal torus $  T_1$.
Recall that $W(T_1)^F=W(T_1)$ since the corresponding automorphism $F$ on $W(T_1)$ is trivial (from the assumption that $G$ is untwisted). We denote the rank of $G$ as $\mathrm{rk}(G)$ and the relative $F$-rank of $T_w$ as $\mathrm{rrk}(T_w)$, cf. \cite[Definition 2.2.11]{geck2020character}.

  \begin{defe}
    For an irreducible character $\chi$ of $W(T_1)$, i.e., $\chi \in \widehat{W(T_1)}$, the corresponding \emph{unipotent almost  character} is 
    \[
    U_\chi \colonequals \frac{1}{|W(T_1)|}\sum_{w\in W(T_1)} \chi(w) R_{T_w}^G(1).\footnote{Note that every unipotent almost  character is a unipotent character when $G=\GL_n$, but this does not hold in general. In addition, recall that $U_{triv}$ is the trivial character $1$ of $G^F$ and $U_{\operatorname{sgn}}$ is the Steinberg character $St$ of $G^F$, where $triv$ is the trivial character and $\operatorname{sgn}$ the sign character of $W(T_1)$.}
    \]
\end{defe}

\subsection{Multiplicity}  \label{intro:multiplicity}
Our first main result is the following.
\begin{thm}\label{thm:main1intro}
Let
$\chi_1, \chi_2, \ldots, \chi_m \in \widehat{W(T_1)}$ and
$\theta_1, \ldots, \theta_n \in \widehat{T_w^F}$ for $w\in W(T_1)$, where $m, n \ge 1$.
Then  we have 
 \begin{equation}\label{eq:mainstregu}
 \begin{split}
&  \left\langle U_{\chi_1}\otimes \cdots \otimes U_{\chi_m}  \otimes R_{T_w}^G(\theta_1)\otimes \cdots \otimes R_{T_w}^G(\theta_n),1\right\rangle\\
= &   \sum_{[\fT^F,u] \in\Xi^{G^F}}
\left(\prod_{i=1}^{m }  \frac{\displaystyle\sum_{\bar{w} \in W(T_1)} \chi_i(\bar{w})  \displaystyle \sum_{\substack{ {v}\in W_{\bar{w}}(\fT) }}Q_{ \tilde{v}^{-1}T_{\bar{w}}\tilde{v} }^{\fT}( u) }{|W(T_1)|}   \right) 
\frac{  1 }{  |G_{{[\fT^F,u]}}^F| } \sum_{\underline{\fT}^F \in [\fT^F]}\frac{ 1}{ |W_w(\fT)| }   \sum_{\substack{ v_i\in W_w(\fT)\\ \text{for }i=1,\ldots , n,\\
u_j\in \underline{\fT}_u^F}}  \prod_{i=1}^nQ_{\tilde{v}_i^{-1}T_w\tilde{v}_i}^{\underline{\fT}}(u_j)\Delta_{\underline{\fT}^F ,\fT'^F}^{\prod_{i=1}^n \dot{v}_i\cdot \theta_i },
\end{split}
\end{equation}
where \begin{enumerate}
  \item[$\bullet$] we sum over $u_j $ each representative of   orbits in $\fT_u^F:=\{u'\in \fT^{uni,F}\,|\, [\fT^F,u']=[\fT^F,u]\text{ in } \Xi^{G^F} \}/\fT^F$ (where  $\fT^{uni,F}$ is the set of unipotent elements in $\fT^F$),
\item[$\bullet$]   ${\Xi}^{G^F}\colonequals \{ (\fT^F,u)\,|\, \fT^F\in G_{pls}^F,\ u\in  \fT^{uni,F}\}/G^F$    with   $ \omega_{G^F}\,:\, G^F\rightarrow \Xi^{G^F} \ \text{given by }  \omega_{G^F}(g)=[G_{g_s}^F, g_u]$,
\item[$\bullet$] $G_{[\fT^F,u]}:=G_{g_{[\fT^F,u]}}$ for any element $g_{[\fT,u]}\in \omega_{G^F}^{-1}([\fT^F,u])$,
\item[$\bullet$] 
 $W_v(\fT):=(W(T_v) /W_{\fT}(T_v))^F$ for $v\in W(T_1)$,\footnote{Here, $W(T_w)/W_{G_{g_s}}(T_w)$ is the quotient space (not a quotient group), and $(W(T_w)/W_{G_{g_s}}(T_w))^F = \{ v W_{G_{g_s}}(T_w) \mid F(v W_{G_{g_s}}(T_w)) = v W_{G_{g_s}}(T_w) \}$.} and $\tilde{v}=\dot{v}\kappa \in N_G(T_w)\fT$,
  \item[$\bullet$] $\mu_w $ is the M\"obius function on the poset $\mathcal{T}_w^{G^F}:=\{ G_s^F\,|\, s\in T_w^F \}$ (with the inclusion partial ordering),
   \item[$\bullet$] $\Delta_{\fT^F,\fT'^F}^{\prod_{i=1}^n \dot{v}_i\cdot \theta_i }:=\sum_{\substack{\fT'^F\in \mathcal{T}_w^{G^F}\\ \fT^F \subset \fT'^F}}\mu_w(\fT^F,\fT'^F)\delta_{\prod_{i=1}^n \dot{v}_i\cdot \theta_i,\fT'}$,
  \item[$\bullet$] 
 $\displaystyle \delta_ {\theta,\fT}:= \begin{cases}
 |  Z(\fT)^F|\quad  &\text{if }  \theta|_{  Z(\fT)^F}=1\\
 0 &\text{otherwise}
 \end{cases}$ for $\theta \in \widehat{T_w^F}.$
\end{enumerate}
\end{thm}
Note that here we used the M\"obius function on the poset $\mathcal{T}_w^{G^F}$. To use this function, in the computation, we only need to consider $\fT^F\in [\fT^F]$ containing $T_w^F$. Therefore, $W_w(\fT)$ and $Q_{\tilde{v}^{-1}T_w\tilde{v}}^{\fT}$ are well-defined. Therefore, for convenience, we consider $Q_{\tilde{v}^{-1}T_w\tilde{v}}^{\fT}$ and $\Delta_{\fT^F,\fT'^F}^{\prod_{j=1}^n \dot{v}_j\cdot \theta_j }$ to be $0$ whenever $\fT$ does not contain $T_w$. The proof is in \S\ref{s:2}.


\begin{rem}
Let us take $\theta_1, \ldots, \theta_n \in \widehat{T_w^F}$ such that no $\theta_i$ is fixed by any non-trivial element of $W(T_w)^F$; such elements are said to be \emph{in general position}, cf. \cite[\S7.3]{carter1985finite}. It is then well known that each character
$\epsilon_G \epsilon_{T_w} R_{T_w}^G(\theta_i)$
is irreducible, where $\epsilon_G = (-1)^{\mathrm{rk}(G)}$ and $\epsilon_{T_w} = (-1)^{\mathrm{rrk}(T_w)}$ (cf.\ \cite[Corollary 2.2.9]{geck2020character}). We refer to such character $R_T^G(\theta)$ (i.e., $\theta$ in general position) as a {regular semisimple character}. Thus, when each $U_{\chi_j}$ is a unipotent character  and each $\theta_i$ is in general position, our result contributes to the computation of the multiplicities of the corresponding  irreducible representations. (Note that there is a discussion regarding the relation between unipotent almost  and actual unipotent characters at MathOverflow \cite{MOflow}.) In addition, our formula can be simplified when $T_w=T_1$ from Remark \ref{rem:split-}. Furthermore, when every unipotent almost   character is trivial and $\theta_1,\ldots, \theta_n$ are split, generic and in general position, the multiplicity can be written by a sum of the size of generic additive character varieties, cf. \cite[Theorem 31]{Nam}. (We can recover \cite[\S4.2]{Nam} from our formula in Theorem \ref{thm:main1intro}.) 
\end{rem}

\subsubsection{Motivation}\label{sss:121}
In \cite{KNP}, the authors have computed the size of the character variety $\{(a_1,b_1,\ldots , a_g,b_g, \\ c_1, \ldots , c_n)\in G^{2g}\times \prod_{i=1}^n C_i \,|\, \prod_{i=1}^g [a_i,b_i]\prod_{j=1}^n c_j=1\}/\!\!/ G$ over a finite field, where $C_1, \ldots, C_n$ are conjugacy classes of regular semisimple or regular unipotent element, with classes of both types present. Then, following the approach in \cite{KNWG}, it becomes a natural question to investigate the corresponding additive analogue
$A_\fg:= \{(X_1,Y_1,\ldots , X_g,Y_g,Z_1, \ldots , Z_n)\in \fg^{2g}\times \prod_{i=1}^n O_i \,|\, \sum_{i=1}^g [X_i,Y_i]+\sum_{j=1}^n Z_j=0\}/\!\!/ G$, where $\fg$ is the Lie algebra of $G$, $O_1, \ldots , O_n$ are adjoint orbits in  $\fg$  of regular semisimple or the closure of regular nilpotent elements, with classes of both types present. Observe that, in the case \(G = \GL_n\), the cardinality of $A_{\mathfrak{g}}$ is related to the multiplicity of the tensor product of irreducible characters of \(\GL_n^F\); see \cite[Equations (1.3.4) and (1.4.1)]{HLRV}.

Let us explain a motivation from the above work for the idea of this paper by considering the Steinberg character and a regular semisimple character $R_T^G(\theta)$. It is a natural idea to relate a regular semisimple conjugacy class with regular semisimple character. Furthermore, if we take the closure of the regular nilpotent orbit in $A_\fg$, then its Fourier transform is related to the Steinberg character from \cite{springer1980steinberg} or \cite[Proposition 3.6]{lehrer1996space} (and recall that the size $|A_\fg^F|$ can be computed using the Fourier transform). Therefore, it is natural to consider the tensor product with the Steinberg character and a regular semisimple character. Starting from this idea, we study the generalised problem of determining the multiplicities of Deligne--Lusztig characters and unipotent almost   characters.

\begin{rem}Let us consider the additive character variety
$A_\fg:= \{(X_1,Y_1,\ldots , X_g,Y_g,Z_1, \ldots , Z_n)\in \fg^{2g}\times \prod_{i=1}^n O_i \,|\, \sum_{i=1}^g [X_i,Y_i]+\sum_{j=1}^n Z_j=0\}/\!\!/ G$ such that $O_1, \ldots , O_n$ are adjoint orbits in  $\fg$  of strongly generic regular semisimple\footnote{Here, strongly means that the centraliser of an element of $O_i$ is connected, and for definition of generic, please see \cite[Definition 11]{Nam}.} or the closure of regular nilpotent elements, with classes of both types present. Then its size $|A_\fg^F|$ can be computed using the Fourier transform on the class functions on $\fg^F$. Briefly, we need to compute the sum
\[
\frac{|Z(G)^F||\fg^F|^{g-1}}{|G^F|} \sum_{x\in \fg^F} |\fg_x^F|^{g} \prod_{i=1}^n \mathcal{F}(1_{O_i}^G)(x),
\]
where $\fg_x$ is the centraliser of $x$ in $\fg$.
Following \cite{KNWG} together with \cite[Proposition 3.6]{lehrer1996space}, this sum can be computed explicitly. The precise value is left to the reader.
\end{rem}

 \bigskip

 \subsection{Frobenius-Schur indicators of the modules over the Drinfeld	double }\label{ss:12} The higher Frobenius–Schur indicators of modules over semisimple Hopf algebras form an interesting topic in Hopf algebras, for example, \cite{IMM, KSZ}. Furthermore, for a finite group $H$, we can consider the Frobenius-Schur indicators of the modules over the Drinfeld
double $D(H)$ of a finite group $H$ via the group algebra $\mathbb{C}[H]$. A key ingredient of this indicator is the cardinality of the set
\[
G_m(g,z):= \{x\in H\,|\, x^m=(gx)^m=z\} \quad \text{for }g,z\in H\text{ and } m\in \mathbb{Z} 
\]
from \cite[\S3]{sch}.

Motivated by this, our next objective is to compute its size by decomposing over $\widehat{G^F}$.
However, it is a difficult problem to treat arbitrary $m$ and $z$.
For this reason, we restrict our attention to the case $$  (m,z)=(|G^F|_{p'},1),$$ where $p$ is the characteristic of $k$. In other words, $|G^F|_{p'}=|G^F|/q^{|\Phi^+|}$ where $\Phi^+$ is the set of positive roots of $G$. 
Then we can get the following result, and this is proved in \S\ref{s:3}.
 \begin{thm} 
For each $g\in G^F$, we have that
\[
| G_{|G^F|_{p'}}(g,1) |= 
  |G^F| \sum_{\chi \in \widehat{G^F}} \frac{\chi(g)}{\chi(1)} \left|  \sum_{[\fT^F ]\in \mathbb{T}^{G^F}}\frac{1}{|G^F| |\fT^F|_p }  \sum_{(T,\theta)}\frac{\epsilon_{\fT}\epsilon_T\left\langle R_T^G(\theta),\chi\right\rangle   }{|(W(T) /W_\fT(T))^F|}\sum_{\underline{\fT}^F\in [\fT^F]}\sum_{\substack{\fT'^F\in  \mathcal{T}_T^{G^F}\\ \underline{\fT}^F\subset \fT'^F}} \mu_T(\underline{\fT}^F,\fT'^F)\delta_{\theta,\fT'^F}\right|^2,\]
  where $(T,\theta)$ runs over all pairs $(T,\theta)$ such that $T\subset G$  is an $F$-stable maximal torus and $\theta\in \widehat{T^F}$, $\mathbb{T}^{G^F}:= \{ [G_s^F]\,|\, s\in G^{ss,F}\}$ and $\mathcal{T}_T^{G^F}:=\{ G_s^F\,|\, s\in T^F\}$ with the M\"obius function $\mu_T$.    \end{thm}
  Note that the term $\left\langle R_T^G(\theta), \chi \right\rangle$ is an integer, and this term is studied by Lusztig, cf. \cite{lusztig}. 
Moreover, when $\theta \in \widehat{T_1^F}$, this value can be computed using the double centraliser theorem, for example, \cite{HL2,eti}.

 \bigskip 
 \subsection{Multiplicity of unipotent almost   characters} From the work of Letellier \cite{letellier2013tensor}, it is an interesting topic to consider $\left\langle U_{\chi_1}\otimes \cdots \otimes U_{\chi_m},1 \right\rangle$ for irreducible characters $\chi_1, \ldots , \chi_m$ of $W(T_1)$.  Recall that unipotent characters of $G^F$ are essential materials to study irreducible characters of $G^F$. However, for a general finite reductive group, unipotent characters remain mysterious to compute every character value. With this observation, we consider unipotent almost   characters, whose every character value is well-known.  The following is the last result of this paper.
 \begin{thm}
 Let $\chi_1, \ldots, \chi_m \in \widehat{W(T_1)}$. Then we have 
 \[
 \left\langle U_{\chi_1}\otimes \cdots \otimes U_{\chi_m},1\right\rangle =\frac{1}{|G^F||W(T_1)|^m} \sum_{\xi=[\fT^F,u] \in \Xi^{G^F}}  |\omega_{G^F}^{-1}(\xi)|\prod_{i=1}^m\left(\sum_{w\in W(T_1)} \chi_i(w)\sum_{v\in W_{w}(\fT)}Q_{{\tilde{v}^{-1}T_{w }\tilde{v}}}^{\fT}(u)\right).
\]
\end{thm}
This is proved in \S\ref{s:4}.

 \bigskip 
  
\section{Multiplicity of Deligne-Lusztig characters and unipotent almost   characters}\label{s:2}

In this section, we compute the multiplicity to prove Theorem \ref{thm:main1intro}:
\begin{equation}\label{eq:2first}
\left\langle U_{\chi_1}\otimes \cdots \otimes U_{\chi_m}\otimes   R_{T_w}^G(\theta_1)\otimes \cdots \otimes R_{T_w}^G(\theta_n),1\right\rangle =\frac{1}{|G^F|} \sum_{g\in G^F} U_{\chi_1}(g) \cdots   U_{\chi_m}(g)  R_{T_w}^G(\theta_1)(g)\cdots R_{T_w}^G(\theta_n)(g),
\end{equation}
where $w \in W(T_1)$,  $\theta_1,\ldots , \theta_n \in \widehat{T_w^F}$, and
$\chi_1 , \ldots, \chi_m \in \widehat{W(T_1)}$.  

\bigskip
\subsection{Types}\label{ss:types}
We define types of elements in $G^F$ in order to decompose Equation~\eqref{eq:2first} into simpler terms. 
Then we may consider pairs $(\mathfrak{T}^F, u)$, where $\mathfrak{T} \in G_{\mathrm{pls}}^F$ (recall that $G_{\mathrm{pls}}^F$ is the set of finite centraliser subgroups $G_s^F$ of semisimple elements $s$ in $G^F$) and $u$ is a unipotent element of $\mathfrak{T}^F$.  Recall that we defined
  ${\Xi}^{G^F}\colonequals \{ (\fT^F,u)\,|\, \fT^F\in G_{pls}^F,\ u \in \fT^{uni,F}\ \}/G^F$ and $ \omega_{G^F}\,:\, G^F\rightarrow \Xi^{G^F} \ \text{given by }  \omega_{G^F}(g)=[G_{g_s}^F, g_u].$

Clearly, $\Xi^{G^F}$ is a finite set. Up to $G^F$-conjugation, every pseudo-Levi subgroup in $G_{pls}^F$ is uniquely determined by its complete root datum with respect to an $F$-stable maximal torus of $G$ (cf. \cite[Definition 1.6.10]{geck2020character}). Furthermore, the number of unipotent conjugacy classes in any such reductive subgroup is finite. Because both the number of $G^F$-conjugacy classes of $F$-stable pseudo-Levi subgroups and the number of unipotent conjugacy classes within them depend only on the underlying root system and Weyl group, the cardinality of the set $\Xi^{G^F}$ is independent of $q$ (except for small $q$).

\bigskip 
\subsection{Vanishing values} \label{ss:22}
From \cite[Theorem 2.2.16 and Example 2.2.17 (a)]{geck2020character}, it is known that $R_{T_w}^G(\theta)(g)=0$ whenever $g_s$ is not conjugate in $G^F$ to any element of $T_w^F$. Therefore, our problem is reduced to compute the following:
\begin{equation*}\label{eq:multioverT}
\begin{split}
 \left\langle U_{\chi_1}\otimes \cdots \otimes U_{\chi_m}\otimes   R_{T_w}^G(\theta_1)\otimes \cdots \otimes R_{T_w}^G(\theta_n),1 \right\rangle &=\frac{1}{|G^F|} \sum_{\substack{g\in G^F\\\text{s.t. } hg_sh^{-1}\in {T_w^F}\\ \text{for some } h\in G^F}} U_{\chi_1}(g) \cdots   U_{\chi_m}(g)   R_{T_w}^G(\theta_1)(g)\cdots R_{T_w}^G(\theta_n)(g).
\end{split}
\end{equation*}
Thus, in this section we only need to consider  those elements $g \in G^F$ whose semisimple part $g_s$ is $G^F$-conjugate to an element of $T_w^F$, and we denote by $\Xi_w^{G^F}$ the subset of such types in $\Xi^{G^F}$, i.e.,  $[\fT^F,u]\in \Xi_w^{G^F}$ if and only if $\fT$ contains $T_w$ up to $G^F$-conjugation with \cite[Proposition 3.5.2]{carter1985finite}.

\bigskip 
  \subsection{Computation} 
Consider the set
$\mathbb{T}_w^{G^F} := \{\, [G_s^F] \mid s \in T_w^F \,\},$
where $[G_s^F]$ denotes the $G^F$-conjugacy class of the centraliser $G_s^F$.
The set $\mathbb{T}_w^{G^F}$ is finite 
and independent of $q$ (except for small $q$), since each element of $\mathbb{T}_w^{G^F}$ is determined by its complete root datum. 
We now define a map
$$\tau_w:=\tau_{\mathbb{T}_w^{G^F}}\,:\,  T_w^F \rightarrow \mathbb{T}_w^{G^F}\quad  \text{given by}\quad \tau_{w}(s)=[G_s^F].$$


\subsubsection{Character values}   Let us consider values $U_\chi(g)$ and $R_{T_w}^G(\theta)(g)$ for an element $g \in G^F$.
\begin{thm}\label{lem:almostunipotentvalue}
For an element $g = g_s g_u = g_u g_s \in G^F$ and a character $\chi \in \widehat{W(T_1)}$, we have
\begin{equation*}\label{eq:prop6}
U_\chi(g) = \frac{1}{|W(T_1)|} \sum_{w \in W(T_1)} \chi(w) R_{T_w}^G(1)(g) = \frac{1}{|W(T_1)|} \sum_{w \in W(T_1)} \chi(w)  \sum_{\substack{\tilde{v}\in (W(T_w)  /W_{G_{g_s}}(T_w))^F }}Q_{ \tilde{v}^{-1}T_w\tilde{v} }^{G_{g_s}}(g_u).
\end{equation*}
  Since the Deligne--Lusztig character value $R_{T_w}^G(1)(g)$ vanishes whenever $g_s$ is not $G^F$-conjugate to an element of $T_w^F$ (cf. \cite[Example 2.2.17 (a)]{geck2020character}), we may restrict the summation over $w \in W(T_1)$ to those elements for which $g_s \in T_w^F$ up to $G^F$-conjugation. For notation, please see Lemma \ref{lem:decompositionofcen} and  Proposition \ref{thm:charactervalue}.
   \end{thm}
Theorem \ref{lem:almostunipotentvalue} implies that  $U_\chi(g_1)=U_\chi(g_2)$ whenever $\omega_{G^F}(g_1)=\omega_{G^F}(g_2)$ for $ g_1,g_2\in  G^F$ since $U_\chi$ is a $G^F$-class function. So let us denote the value $\displaystyle \frac{1}{|W(T_1)|} \sum_{w \in W(T_1)} \chi(w)  \sum_{\substack{\tilde{v}\in (W(T_w)  /W_{G_{g_s}}(T_w))^F }}Q_{ \tilde{v}^{-1}T_w\tilde{v} }^{G_{g_s}}(g_u)  
$ by $$\displaystyle \frac{1}{|W(T_1)|} \sum_{w \in W(T_1)} \chi(w)  \sum_{\substack{\tilde{v}\in (W(T_w)  /W_{\fT}(T_w))^F }}Q_{ \tilde{v}^{-1}T_w\tilde{v} }^{\fT}( u)  
$$ via $\omega_{G^F}(g)=[\fT^F,u]$.
Theorem  \ref{lem:almostunipotentvalue} follows from the following lemma and proposition.

\begin{lem}\label{lem:decompositionofcen}

    For a semisimple element $s$ in $T_w^F$, we have $$\{ x\in G^F \,|\, x^{-1}sx\in T_w^F\}=\underset{v\in (W(T_w)  /W_{G_s}(T_w))^F }{\sqcup} ( \tilde{v} G_s^F)^{-1}, $$ 
     where $\tilde{v}= \dot{v}\kappa\in G^F$ for $\kappa\in G_s$ and $\dot{v}$ is a representative of $v  $ in $N_{G }(T_w)$, and $H^{-1}$ is the set of inverses of elements     of $ H$.
\end{lem}
From now on, we use $\tilde{v}$ and $\dot{v}\kappa$ for $v\in (W(T_w)/W_\fT(T_w))^F$ in this paper without recalling.
\begin{proof}
Let us show that $\{ x\in G^F \,|\, x^{-1} sx \in T_w^F\}\subseteq \underset{v\in (W(T_w)   /W_{G_s}T_w)^F }{\sqcup} (\tilde{v} G_s^F )^{-1}.$ Since $x^{-1}sx \in T_w^F\subset T_w$, we can check that $x T_wx^{-1}, T_w\subset G_s$. Then there exists $h\in G_s$ such that $hxT_wx^{-1}h^{-1}=T_w$. Note that we have $F(x^{-1}h^{-1}shx )=F(x^{-1}sx )=x^{-1}sx =x^{-1}h^{-1}shx $. This implies that $(x^{-1}h^{-1})^{-1}F(x^{-1}h^{-1})sF(hx )(hx )^{-1}=s$, and so we get that $(x^{-1}h^{-1})^{-1}F(x^{-1}h^{-1})\in G_s$. Since $(hx)^{-1}=x^{-1}h^{-1}\in N_G(T_w)$, we have
\[
(x^{-1}h^{-1})^{-1}F(x^{-1}h^{-1})\in N_{G_s}(T_w).
\]
From this, we can check the following: $F(x^{-1}h^{-1}N_{G_s}(T_w))=F(x^{-1}h^{-1}) N_{G_s}(T_w)=x^{-1}h^{-1}N_{G_s}(T_w)$. Furthermore, by quotient with $T_w$, we have that
\[
F(x^{-1}h^{-1}W_{G_s}(T_w))=x^{-1}h^{-1}W_{G_s}(T_w)\Rightarrow x^{-1}h^{-1}\in (W(T_w)/W_{G_s}(T_w))^F.
\]
(For convenience, here we use $x,h$ as the image of $x,h $ in $W(T_w)$.)
From $(x^{-1}h^{-1})^{-1}F(x^{-1}h^{-1})\in G_s$, we can  find $h'\in G_s$ such that $x^{-1}h^{-1}h' \in G^F$ via Lang's theorem. Then it is easy to see that $h^{-1}h'\in G_s^F$ due to $x\in G^F$. Since $x^{-1},h^{-1}h'\in G^F$, we can consider the following (with abuse of notation) 
\[
x^{-1}h^{-1}h'\in (W(T_w)/W_{G_s}(T_w))^F\cdot G_s\Rightarrow x^{-1}=\dot{v}\kappa h'^{-1}h\in (W(T_w)/W_{G_s}(T_w))^F\cdot G_s \cdot G_s^F,
\]
where $\dot{v}\kappa\in G^F$ (with $\dot{v}\in N_{G }(T_w)$ and $\kappa \in G_s$).
Note that if two distinct elements $x $ and $y$ are in the same coset, then we can choose same representative $\dot{v}$. This is because $\dot{v}_xW_{G_s}(T_w)=\dot{v}_yW_{G_s}(T_w)$, then $\dot{v}_x=\dot{v}_y m $ for some $m\in G_s$.
 Therefore, by taking $\tilde{v}$ as $\dot{v}\kappa$, we are done.

To show the converse inclusion, let us consider $\tilde{v}h=\dot{v}uh$ for some $h\in G_s^F$. Then it is easy to check that
\[
\dot{v}\kappa h s (\dot{v}\kappa h)^{-1}=\dot{v}s\dot{v}^{-1}.
\]
From the condition on ${v}$, we have $\dot{v}^{-1}F(\dot{v})\in N_{G_s}(T_w)$, and so we have $F(\dot{v})=\dot{v}m$ for some $m\in N_{G_s}(T_w)$. Then we have
\[
F(\dot{v}s\dot{v}^{-1})=F(\dot{v})sF(\dot{v})^{-1}=\dot{v}msm^{-1}\dot{v}^{-1}=\dot{v}s\dot{v}^{-1}.
\]
Therefore, we can conclude that $x^{-1}sx=\dot{v}\kappa h s (\dot{v}\kappa h)^{-1}=\dot{v}s\dot{v}^{-1}\in T_w^F$, and so this finishes the proof.
\end{proof}

 \begin{prop}\label{thm:charactervalue}
    For   $\theta \in \widehat{T_w^F}$, we have
   \[
    \begin{split}
    R_{T_w}^G(\theta)(g) = \sum_{\substack{ {v}\in (W(T_w)  /W_{G_{g_s}}(T_w))^F }}Q_{ \tilde{v}^{-1}T_w\tilde{v} }^{G_{g_s}}(g_u)\theta(\dot{v}g_s\dot{v}^{-1}).
    \end{split}
    \]
\end{prop}
Note that $\tilde{v}^{-1}T_w\tilde{v}$ is $G^F$-conjugate to $T_w$ (since $\tilde{v}\in G^F$), but it may not be $G_{g_s}^F$-conjugate to $T_w$. For example, let us consider $G^F:=\mathrm{GL}_4^F$ and $\fT^F=\mathrm{GL}_2^F\times \mathrm{GL}_2^F=C_{G^F}(\mathrm{diag}(a,a,b,b))$. Then two non-split maximal tori $T_{(1,2)}$ and $T_{(3,4)}$ are $G^F$-conjugate, but they are not over $\fT^F$.

\begin{proof}We have 
\[    \begin{split}
    R_{T_w}^G(\theta)(g)&=\frac{1}{|G_{g_s}^F|}\sum_{\substack{x\in G^F \text{ s.t.}\\ x^{-1}g_sx\in T_w^F}}Q_{xT_wx^{-1}}^{G_{g_s}}(g_u)\theta(x^{-1}g_sx)\\
    &=\frac{1}{|G_{g_s}^F|}\sum_{\substack{x\in \underset{v\in (W(T_w)  /W_{G_{g_s}}(T_w))^F }{\sqcup} ( \tilde{v} G_{g_s}^F)^{-1}}}Q_{ x T_wx^{-1} }^{G_{g_s}}(g_u)\theta(x^{-1}g_sx)
    \\
    &= \sum_{\substack{ {v}\in (W(T_w)  /W_{G_{g_s}}(T_w))^F }}Q_{ \tilde{v}^{-1}T_w\tilde{v} }^{G_{g_s}}(g_u)\theta(\dot{v}g_s\dot{v}^{-1}),
    \end{split}
    \]
    where the first equality is \cite[Theorem  2.2.16]{geck2020character} and the second equality is Lemma \ref{lem:decompositionofcen}.
    Let us show the last equality. Let us take any two distinct elements $x$ and $y$ in $(\tilde{v}G_{g_s}^F)^{-1}$, and each has decomposition $\dot{v}u_x h_x$ and $\dot{v}u_yh_y$ for $u_x,u_y\in G_{g_s}$ and $h_x,h_y\in G_{g_s}^F$.
  (Recall that we can assume that $x$ and $y$ have same representative $\dot{v}$ from the proof of Lemma \ref{lem:decompositionofcen}.) Then our goal is to show that $Q_{h_x^{-1}u_x^{-1}\dot{v}T_w \dot{v}u_xh_x}^{G_{g_s}}=Q_{h_y^{-1}u_y^{-1}\dot{v}T_w \dot{v}u_yh_y}^{G_{g_s}}.$ From \cite[Definition 2.2.15]{geck2020character} and $\dot{v}\in N_G(T_w)$, our problem is reduced to prove the following equality
    $$Q_{u_x^{-1} T_w u_x}^{G_{g_s}}=Q_{u_y^{-1} T_w  u_y }^{G_{g_s}}.$$
    Now, let us consider 
    \[
    xy^{-1}=h_x^{-1}u_x^{-1}\dot{v}\dot{v}^{-1}u_yh_y=h_x^{-1}u_x^{-1}u_yh_y.
    \]
    From $xy^{-1},h_x,h_y\in G^F$, we can get that $u_x^{-1}u_y\in G_s^F$. Then we have the following:
    $$Q_{u_x^{-1} T_w u_x}^{G_{g_s}}=Q_{u_y^{-1}u_xu_x^{-1} T_w u_xu_x^{-1}u_y}^{G_{g_s}}=Q_{u_y^{-1} T_w  u_y }^{G_{g_s}}$$
    using the property in \cite[Definition 2.2.15]{geck2020character} again. Therefore, we are done.
 \end{proof}

\begin{proof}[Proof of Theorem \ref{lem:almostunipotentvalue}]
Note that $R_{T_w}^G(1)(g)=0$ if $g_s$ is not in $T_w^F$ up to $G^F$-conjugation. Therefore, let us assume that $g_s$ is $G^F$-conjugate to an element in $T_w^F$. Without   loss of generality,  we can assume that $g_s\in T_w^F$ since $R_{T_w}^G(1)$ is a class function on $G^F$, i.e., $U_\chi$ is also a $G^F$-class function.
Then we have 
\[\begin{split}
\frac{1}{|W(T_1)|}\sum_{w\in W(T_1)} \chi(w) R_{T_w}^G(1)(g)&=\frac{1}{|W(T_1)|}\sum_{w\in W(T_1)} \chi(w) \sum_{\substack{\tilde{v}\in (W(T_w)  /W_{G_{g_s}}(T_w))^F }}Q_{ \tilde{v}^{-1}T_w\tilde{v} }^{G_{g_s}}(g_u)1(\dot{v}g_s\dot{v}^{-1})\\
&=\frac{1}{|W(T_1)|}\sum_{w\in W(T_1)} \chi(w) \sum_{\substack{\tilde{v}\in (W(T_w)  /W_{G_{g_s}}(T_w))^F }}Q_{ \tilde{v}^{-1}T_w\tilde{v} }^{G_{g_s}}(g_u). 
\end{split}
\]
where the first equality comes from Proposition \ref{thm:charactervalue}.
\end{proof}
\begin{rem}\label{rem:split-}
When $T_w$ is split, i.e., $w=1$,   we have $\{ x\in G^F \,|\, xsx^{-1}\in T_w^F\}=\underset{v\in W(T) /W_{G_s}(T) }{\sqcup}\dot{v} G_s^F  $ with $\dot{v}\in N_G(T)^F$ since the Frobenius action on $W(T)$ and  $W_{G_s}(T)$ are trivial. Then this implies that $R_{T_1}^G(\theta)(g)= \frac{Q_{ T_1 }^{G_{g_s}}(g_u)}{|W_{G_{g_s}}(T_1)|}\sum_{\substack{v\in W(T_1) }}\theta(\dot{v}g_s\dot{v}^{-1})$.
 
\end{rem}

\subsubsection{$n=1$ case} 
With Theorem \ref{lem:almostunipotentvalue}, we have the following:
\begin{equation}\label{eq:222}\begin{split}
&\left\langle U_{\chi_1}\otimes \cdots \otimes U_{\chi_m}\otimes   R_{T_w}^G(\theta),1\right\rangle \\=&\frac{1}{|G^F|} \sum_{[\fT^F,u] \in \Xi_w^{G^F}}
\left(\prod_{i=1}^{m }  \frac{\displaystyle\sum_{{\bar{w}} \in W(T_1)} \chi_i({\bar{w}})  \displaystyle \sum_{\substack{\tilde{v}\in (W(T_{\bar{w}})  /W_{\fT}(T_{\bar{w}}))^F }}Q_{ \tilde{v}^{-1}T_{\bar{w}}\tilde{v} }^{\fT}( u) }{|W(T_1)|}   \right)  
\sum_{\left\{ hgh^{-1} \,\middle\vert\, \substack{
h\in G^F,\ g\in \omega_{G^F}^{-1}([\fT^F,u])\\\text{s.t. } g_s\in \tau_w^{-1}([\fT^F])
}\right\}}   R_{T_w}^G(\theta)(hgh^{-1}),
\end{split}
\end{equation}
where ${\left\{ hgh^{-1} \,\middle\vert\, \substack{
h\in G^F,\ g\in \omega_{G^F}^{-1}([\fT^F,u])\\\text{s.t. } g_s\in \tau_w^{-1}([\fT^F])
}\right\}} $ is the set of all elements $G^F$-conjugate to $g$ such that the semisimple part of $g$   is in $ \tau_w^{-1}([\fT^F])$.

Now, our problem is reduced to compute the sum $\displaystyle \sum_{\left\{ hgh^{-1} \,\middle\vert\, \substack{
h\in G^F,\ g\in \omega_{G^F}^{-1}([\fT^F,u])\\\text{s.t. } g_s\in \tau_w^{-1}([\fT^F])
}\right\}}   R_{T_w}^G(\theta)(hgh^{-1})$. Since $R_{T_w}^G(\theta)$ is a class function (over $G^F$-conjugation),  it suffices to sum over representatives of elements in  $\tau_w^{-1}([\fT^F]) / W (T_w)$, i.e., the representatives of the set of conjugacy classes of every element $g_s$ in $\tau_w^{-1}([\fT^F])$. Then we have the following reduction:§
\[
\begin{split}
\sum_{\left\{ hgh^{-1} \,\middle\vert\, \substack{
h\in G^F,\ g\in \omega_{G^F}^{-1}([\fT^F,u])\\\text{s.t. } g_s\in \tau_w^{-1}([\fT^F])
}\right\}}  R_{T_w}^G(\theta)(hgh^{-1})&=\frac{|G^F|}{|G_{g}^F|} 
\sum_{\substack{u_j\in \fT_u^F}} \sum_{\substack{
 g_s\in \tau_w^{-1}([\fT^F])  / W (T_w) }}  R_{T_w}^G(\theta)(g_su_j),
\end{split}
\]
where we sum over $u_j $ each representative of   orbits in $\fT_u^F$, and the term $ \tau_w^{-1}([\fT^F])  / W (T_w)$ is an equivalence relation coming from $W(T_w)$-action on $T_w$.

We can show  this by proving that a product map $\tau_w^{-1}([\fT^F])/W(T_w) \times \fT_u^F \rightarrow   \left\{ hgh^{-1} \,\middle\vert\, \substack{
h\in G^F,\ g\in \omega_{G^F}^{-1}([\fT^F,u])\\\text{s.t. } g_s\in \tau_w^{-1}([\fT^F])
}\right\}/G^F$ is bijective, where we choose representatives of elements in $\tau_w^{-1}([\fT^F])/W(T_w) $ satisfying that its centraliser is $\fT$. (Recall that if two elements in $T$ are conjugate, they are actually $W(T)$-conjugate.)  Surjectivity is obvious from the Jordan decomposition, and so let us show that this is injective. If $ysu_jy^{-1}=su_j$ for some $y\in G^F$, then $y\in \fT^F$. This implies that this product map is injective. 
Then from   Lemma \ref{lem:decompositionofcen} and Proposition \ref{thm:charactervalue} , we have  
\begin{equation}\label{eq:DLcharacter-with-semisimple}
\begin{split}
R_{T_w}^G(\theta)(g_su_j)&=\sum_{v\in  (W(T_w) /W_{G_{g_s}}(T_w))^F }Q_{\tilde{v}^{-1}T_w\tilde{v}}^{\fT}(u_j)\theta((\dot{v}\kappa h)g_s(\dot{v}\kappa h)^{-1})\\
&=\sum_{v\in  (W(T_w)  /W_{G_{g_s}}(T_w))^F }Q_{\tilde{v}^{-1}T_w\tilde{v}}^{\fT}(u_j)\theta(\dot{v} g_s\dot{v}^{-1})
\end{split}
\end{equation}
with the same notation in Lemma \ref{lem:decompositionofcen}. Since $\dot{v}\in N_G(T_w)$, we can consider $\dot{v}$ acts on $\theta$. So by defining $\dot{v}\cdot \theta(g_s):=\theta(\dot{v} g_s\dot{v}^{-1})$ (i.e., $\dot{v}\cdot \theta$ is a character of $T_w^F$),
we can write as follows:
\[
R_{T_w}^G(\theta)(g_su_j ) =\sum_{v\in  (W(T_w)  /W_{G_{g_s}}(T_w))^F }Q_{\tilde{v}^{-1}T_w\tilde{v}}^{\fT}(u_j)(\dot{v}\cdot \theta)(g_s).
\]

Then this gives the following relation:
\[\begin{split}
\sum_{\substack{
 g_s\in \tau_w^{-1}([\fT^F])  / W (T_w)}}  R_{T_w}^G(\theta)(g_su_j)&= 
\sum_{g_s\in \tau_w^{-1}([\fT^F])  / W(T_w) }  \sum_{v\in  (W(T_w)  /W_{G_{g_s}}(T_w))^F}Q_{\tilde{v}^{-1}T_w\tilde{v}}^{\fT}(u_j)(\dot{v}\cdot \theta)(g_s) \\
&=  \sum_{g_s\in \tau_w^{-1}([\fT^F]) }\frac{1}{|(W(T_w) /W_{G_{g_s}}(T_w))^F|}   \sum_{v\in  (W(T_w) /W_{G_{g_s}}(T_w))^F } Q_{\tilde{v}^{-1}T_w\tilde{v}}^{\fT}(u_j)(\dot{v}\cdot \theta)(g_s),
\end{split}
\]
where   the term $\frac{1}{|(W(T_w) /W_{G_{g_s}}(T_w))^F|}$ comes from the size of each orbit in $ \tau_w^{-1}([\fT^F])  / W(T_w) $ and Lemma \ref{lem:decompositionofcen}.

\subsubsection{}
Let us consider the M\"obius function $\mu_w $ on the poset $\mathcal{T}_w^{G^F}:=\{ G_{g_s}^F\,|\, g_s\in T_w^F \}$ partially ordered by inclusion.
Using this M\"obius function, we can compute the following:\[
\begin{split}
&\sum_{\substack{
g_s\in \tau_w^{-1}([\fT^F])  / W (T_w) }}  R_{T_w}^G(\theta)(g_su_j)\\
& = \sum_{\underline{\fT}^F\in [\fT^F]}  \frac{1}{|(W(T_w) /W_{\underline{\fT}}(T_w))^F|}  \sum_{v\in (W(T_w)  /W_{\underline{\fT}}(T_w))^F } Q_{\tilde{v}^{-1}T_w\tilde{v}}^{\underline{\fT}}(u_j) \sum_{\substack{ g_s\in  T_w^F \\ \text{s.t. }G_{g_s}^F=\underline{\fT}^F}} (\dot{v}\cdot \theta)(g_s)\\
& = \sum_{\underline{\fT}^F\in [\fT^F]}  \frac{1}{|(W(T_w) /W_{\underline{\fT}}(T_w))^F|}   \sum_{v\in (W(T_w) /W_{\underline{\fT}}(T_w))^F }Q_{\tilde{v}^{-1}T_w\tilde{v}}^{\underline{\fT}}(u_j) \left(\sum_{\substack{\fT'^F\in \mathcal{T}_w^{G^F}\\ \underline{\fT}^F\subset \fT'^F}} \mu_w(\underline{\fT}^F,\fT'^F)  \sum_{\substack{g_s\in T_w^F\text{ s.t.} \\ G_{g_s}^F\supset \fT'^F}} (\dot{v}\cdot \theta)(g_s)\right).
\end{split}
\]
\begin{rem}
Note that if $\fT'^F\in [\fT^F]$ such that $\fT'$ does not contain $T_w$, then the corresponding value vanishes. Therefore, we automatically choose elements in $[\fT^F]$ containing $T_w^F$ as noted in after Theorem \ref{thm:main1intro}.
\end{rem}
 
We can verify $\{g_s\in T_w^F\,|\, G_{g_s}^F\supset \fT'^F\}=Z(\fT'^F)$ for each  $\fT'^F\in \mathcal{T}_w^{G^F}$, 
and thus we obtain the following result for the last term:
 $$\sum_{\substack{g_s\in T_w^F \\  \fT'^F\subset G_{g_s}^F }} (\dot{v}\cdot \theta)(g_s)=\sum_{\substack{g_s\in T_w^F \cap  Z(\fT'^F)}} (\dot{v}\cdot \theta)(g_s) =\sum_{\substack{g_s\in  Z(\fT')^F }} (\dot{v}\cdot \theta)(g_s) 
 $$ since $Z(\fT')^F=Z(\fT'^F)$ when $\fT'$ is connected. Then we obtain the following:
\begin{equation*}\label{eq:delta_one}
\delta_{\dot{v}\cdot \theta,\fT'}:= \sum_{\substack{g_s\in  Z(\fT')^F }}  (\dot{v}\cdot \theta)(g_s) =\begin{cases}
 |  Z(\fT')^F|\quad  &\text{if }\dot{v}\cdot \theta|_{  Z(\fT')^F}=1\\
 0 &\text{otherwise}.
 \end{cases}
\end{equation*}
In summary, we have the following:
\begin{lem}
 For a given type $[\fT^F,u]$, we have 
 \[\begin{split}
&\sum_{\left\{ hgh^{-1} \,\middle\vert\, \substack{
h\in G^F,\ g\in \omega_{G^F}^{-1}([\fT^F,u])\\\text{s.t. } g_s\in \tau_w^{-1}([\fT^F])
}\right\}}   R_{T_w}^G(\theta)(hgh^{-1})\\
&=\frac{|G^F| }{|G_{{[\fT^F,u]}}^F|  }    \sum_{\underline{\fT}^F \in [\fT^F]} \frac{1}{|W_w(\underline{\fT})|}  \sum_{v\in W_w(\underline{\fT}) }  \sum_{\substack{u_j\in \underline{\fT}_u^F  }} Q_{\tilde{v}^{-1}T_w\tilde{v}}^{\underline{\fT}}(u_j) \sum_{\substack{\fT'^F\in \mathcal{T}_w^{G^F}\\ \underline{\fT}^F \subset \fT'^F}} \mu_w(\fT^F,\fT'^F)  \delta_{\dot{v}\cdot \theta,\fT'},
\end{split}
 \]
 where we sum over $u_j $ each representative, $G_{[\fT^F,u]}:=G_{g_{[\fT^F,u]}}$ for an element $g_{[\fT^F,u]}$ in $\omega_{G^F}^{-1}([\fT^F,u])$ and $W_w(\fT):=(W(T_w) /W_{\fT}(T_w))^F$ for $w\in W(T_1)$.
 \end{lem}
 Then we can conclude our first main result (for the $n=1$ case) by applying this result to Equation \eqref{eq:222}.
 \begin{thm}\label{thm:main-one}
 We have 
 \[\begin{split}
&\left\langle U_{\chi_1}\otimes \cdots \otimes U_{\chi_m}\otimes  R_{T_w}^G(\theta),1\right\rangle\\
= & \sum_{[\fT^F,u] \in\Xi_w^{G^F}}
\left(\prod_{i=1}^{m }  \frac{\displaystyle\sum_{{\bar{w}} \in W(T_1)} \chi_i({\bar{w}})  \displaystyle \sum_{\substack{ {v}\in W_{\bar{w}}(\fT) }}Q_{ \tilde{v}^{-1}T_{\bar{w}}\tilde{v} }^{\fT}( u) }{|W(T_1)|}   \right)  
\frac{ 1}{  |G_{{[\fT^F,u]}}^F|}  \sum_{\underline{\fT}^F\in [\fT^F]} \frac{ 1}{ |W_w(\fT)|  }    \sum_{\substack{v\in W_w(\underline{\fT}), \\ u_j\in \underline{\fT}_u^F }} Q_{\tilde{v}^{-1}T_w\tilde{v}}^{\underline{\fT}}(u_j)\Delta_{\underline{\fT}^F,\fT'^F}^{\dot{v}\cdot \theta}
\end{split}
\] 
where we sum over $u_j $ each representative in $\underline{\fT}_u^F$ and
\[
\Delta_{\fT^F,\fT'^F}^{\dot{v}\cdot \theta}:=\sum_{\substack{\fT'^F\in \mathcal{T}_w^{G^F}\\ \fT^F \subset \fT'^F}} \mu_w(\fT^F,\fT'^F)  \delta_{\dot{v}\cdot \theta,\fT'}.
\]
\end{thm}

\subsection{Multiple Deligne-Lusztig characters}
Now, let us finish the proof of Theorem \ref{thm:main1intro} by considering multiple Deligne-Lusztig characters. Then from Equation \eqref{eq:DLcharacter-with-semisimple}, we have  \[
R_{T_w}^G(\theta_1)(g)\cdots R_{T_w}^G(\theta_n)(g)= \sum_{\substack{v_i\in W_w(\fT) \\  i=1,\ldots, n}}\prod_{i=1}^n Q_{\tilde{v}_i^{-1}T_w\tilde{v}_i}^{\fT}(g_u)(\dot{v}_i\cdot \theta_i)(g_s).
\]
Then with the same computation, we can check that 
 \begin{equation*}\label{eq:multiple}
 \begin{split}
&\left\langle U_{\chi_1}\otimes \cdots \otimes U_{\chi_m}\otimes  R_{T_w}^{G}(\theta_1)\otimes \cdots \otimes R_{T_w}^G(\theta_n),1\right\rangle\\
= &  \sum_{[\fT^F,u] \in\Xi_w^{G^F}}
\left(\prod_{i=1}^{m }  \frac{\displaystyle\sum_{{\bar{w}} \in W(T_1)} \chi_i({\bar{w}})  \displaystyle \sum_{\substack{ {v}\in W_{\bar{w}}(\fT) }}Q_{ \tilde{v}^{-1}T_{\bar{w}}\tilde{v} }^{\fT}( u) }{|W(T_1)|}   \right) 
\frac{  1 }{   |G_{{[\fT^F,u]}}^F| } \sum_{\underline{\fT}^F \in [\fT^F]} \frac{ 1}{ |W_w(\underline{\fT})|  }   \sum_{\substack{ v_i\in W_w(\underline{\fT})\\ \text{for }i=1,\ldots , n,\\
u_j\in \underline{\fT}_u^F  }}  \prod_{\substack{i=1 }}^nQ_{\tilde{v}_i^{-1}T_w\tilde{v}_i}^{\underline{\fT}}(u_j)\Delta_{\underline{\fT}^F,\fT'^F}^{\prod_{i=1}^n \dot{v}_i\cdot \theta_i},
\end{split}
\end{equation*}
where we sum over $u_j $ each representative in $\underline{\fT}_u^F$ and 
\[
\Delta_{\fT^F,\fT'^F}^{\prod_{i=1}^n \dot{v}_i\cdot \theta_i }:=\sum_{\substack{\fT'^F\in \mathcal{T}_w^{G^F}\\ \fT^F \subset \fT'^F}}\mu_w(\fT^F,\fT'^F)\delta_{\prod_{i=1}^n \dot{v}_i\cdot \theta_i,\fT'}
\]
This result gives Theorem \ref{thm:main1intro}, since the terms corresponding to elements in $\Xi^{G^F} \setminus \Xi_w^{G^F}$ vanish by the discussion in \S\ref{ss:22}.

\subsubsection{} An important question is to determine conditions under which this multiplicity is zero or non-zero.
We provide a partial answer for the vanishing case.
\begin{cor}\label{coro:12} Let us assume that $Z(G)^F\neq \{1\}$.
If $\prod_{j=1}^n \dot{v}_i\cdot \theta_i$ is non-trivial on every non-trivial subgroup of $T_w^F$, then the multiplicity vanishes.
\end{cor}
  When $ {T}_w^F$ is cyclic, i.e., $ {T}_w^F\simeq \mathbb{F}_{q^m}^\times$, we can construct an easy example which is non-trivial on every non-trivial subgroup  by taking $\theta\in \widehat{T_w^F}$ such that $|\theta(\alpha)|=q^m-1$, where $\alpha$ is a generator of $T_w^F$.
\begin{proof}
From the definition of $\delta_{\theta,\fT}$, the result follows.
\end{proof}

\subsection{The multiplicity involving the Steinberg character}\label{ss:st-including} Let us assume that $\chi_m$ is the sign character of $W(T_1)$. Then $U_{\chi_m}$ is the Steinberg character of $G^F$, whose character value is zero on every non-semisimple element, cf. \cite[\S7]{DL76}. Hence we only need to consider types $[\fT^F,1]$ in $\Xi_w^{G^F}$.
Under this situation, we provide a partial answer for the non-vanishing case. 

\begin{cor} \label{coro:degree-non-zero-condition} Let us assume that $q$ is sufficiently large and $\chi_m=sgn$.
Let $\Phi(\fT)^+$ be the set of positive roots of $\fT$. If  $$(n-1)(|\Phi^+ | - |\Phi(\fT)^+ |)-(m-1)|\Phi(\fT)^+|+\mathrm{dim} (Z(G))-\mathrm{dim} (Z(\fT)) >0$$ for every $\fT $ containing $T_w$, then $\langle U_{\chi_1}\otimes \cdots \otimes U_{\chi_m}\otimes  R_{T_w}^G(\theta_1)\otimes \cdots \otimes R_{T_w}^G(\theta_n),1 \rangle\neq 0$ unless every $\prod_{j=1}^n \dot{v}_j \cdot \theta_j$ is non-trivial on $Z(G)^F$.
 \end{cor}
 \begin{proof}To find a condition which implies $\langle U_{\chi_1}\otimes \cdots \otimes U_{\chi_m}\otimes  R_{T_w}^G(\theta_1)\otimes \cdots \otimes R_{T_w}^G(\theta_n),1 \rangle\neq 0$, let us consider $q$ as a variable, not a power of prime. Note that we can consider the degree $\langle U_{\chi_1}\otimes \cdots \otimes U_{\chi_m}\otimes  R_{T_w}^G(\theta_1)\otimes \cdots \otimes R_{T_w}^G(\theta_n),1 \rangle$ in $q$.
Then for any pseudo-Levi subgroup $\fT$ of $G$, the degree of $   \frac{ \sum_{{\bar{w}} \in W(T_1)} \chi({\bar{w}})    \sum_{\substack{ {v}\in W_{\bar{w}}(\fT) }}Q_{ \tilde{v}^{-1}T_{\bar{w}}\tilde{v} }^{\fT}( 1) }{|W(T_1)|}  $ in $q$ can be at most $|\Phi(\fT)^+|$ (since $Q_{ T_v }^{\fT}(1)=\epsilon_\fT \epsilon_{T_v} |\fT^F:T_v^F|_{p'}$), and at least $0$.
 
 Let us consider the degree of terms over $[\fT,1]
$ in $ \Xi_w^{G^F}$ of $\left\langle U_{\chi_1}\otimes \cdots \otimes U_{\chi_m}\otimes  R_{T_w}^G(\theta_1)\otimes \cdots \otimes R_{T_w}^G(\theta_n),1\right\rangle$ from the result in \S\ref{eq:multiple}. (Recall that we do not need to consider non-trivial unipotent element from the existence of the Steinberg character.)
When $\fT=G$, the possible smallest degree is $$(n-1)|\Phi^+|-\mathrm{rk}(G)+\mathrm{dim}(Z(G))$$ by considering the degree of $\prod_{i=1}^{m-1}  \frac{ \sum_{{\bar{w}} \in W(T_1)} \chi_i({\bar{w}})    \sum_{\substack{ {v}\in W_{\bar{w}}(G) }}Q_{ \tilde{v}^{-1}T_{\bar{w}}\tilde{v} }^{G}( 1) }{|W(T_1)|}  $ as $0$. When $\fT\neq G$,   the possible largest degree is $$(n-1)|\Phi(\fT)^+|+(m-1)|\Phi(\fT)^+|-\mathrm{rk}(\fT)+\mathrm{dim}(Z(\fT))$$ because the maximal degree of $\prod_{i=1}^{m-1} \frac{ \sum_{{\bar{w}} \in W(T_1)} \chi_i({\bar{w}})    \sum_{\substack{ {v}\in W_{\bar{w}}(\fT) }}Q_{ \tilde{v}^{-1}T_{\bar{w}}\tilde{v} }^{\fT}( 1) }{|W(T_1)|}  $ can be $(m-1)|\Phi(\fT)^+|$. Note that $\mathrm{rk}(G)=\mathrm{rk}(\fT)$. Our assumption then implies that, among the terms of $\Xi_w^{G^F}$ indexed by
pairs $[\mathfrak{T},1]$, the term of highest degree is the one corresponding
to $[\mathfrak{T},1] = [G,1]$; every other term, corresponding to
$[\mathfrak{T},1]$ with $\mathfrak{T}$ a proper subgroup of $G$, has strictly
smaller degree. Consequently, the term corresponding to $[G,1]$ cannot be
cancelled, which shows that the multiplicity is nonzero.\end{proof}
 
 \subsubsection{Deligne-Lusztig characters over a split maximal torus}\label{sss:split-case-simple-result} Let us now consider the case of a split maximal torus $T:=T_1$ and $\chi_m=sgn$.
In this situation, the multiplicity
$\left\langle
U_{\chi_1}\otimes \cdots \otimes U_{\chi_m}
\otimes
R_{T}^G(\theta_1)\otimes \cdots \otimes R_{T}^G(\theta_n),
\, 1
\right\rangle$
admits a simpler expression, given in terms of a linear combination of the dimensions of the unipotent almost   characters of $\mathfrak{T}^F$. Let us assume that $s\in T^F$ and $G_s^F=\fT^F$. Then  from \cite[Example 3.4]{geck2025}, we have $$U_\chi(s)=   \sum_{\psi \in \widehat{W_\fT(T )}} m(\psi,\chi ) \dim\left(U_\psi^{\fT} \right)
 ,$$
 where $U_\psi^{\fT}$ is the unipotent almost   character corresponding to $\psi \in \widehat{W_\fT(T )}$ and $m(\psi,\chi)$ is the multiplicity of $\psi \in \widehat{W_{\fT}(T)}$ in the restriction of $\chi $ to ${W_{\fT}(T)}$.
Then we have the following result with Remark \ref{rem:split-}:
\[
 \begin{split}
&  \left\langle U_{\chi_1}\otimes \cdots \otimes U_{\chi_m}  \otimes R_{T }^G(\theta_1)\otimes \cdots \otimes R_{T }^G(\theta_n),1\right\rangle\\
= &  \sum_{\substack{[\fT^F,u] \in \Xi^{G^F}\\\text{s.t. }u=1}}
\left(\prod_{i=1}^{m }  \sum_{\psi \in \widehat{W_\fT(T )}} m(\psi,\chi_i) \dim \left(U_\psi^{\fT} \right)
\right)
\frac{|W_\fT(T ) |  Q_{T_1}^{\fT}(1)^n}{ |W (T ) ||G_{{[\fT^F,1]}}^F| }  \sum_{\underline{\fT}^F \in [\fT^F]} \sum_{\substack{ v_i\in W(T)/W_{\underline{\fT}}(T)\\ i=1,\ldots , n}} \Delta_{\underline{\fT}^F,\fT'^F}^{\prod_{i=1}^n \dot{v}_i\cdot \theta_i }.
\end{split}
\]
Furthermore, in this case, the M\"obius function $\mu_1$ is studied in \cite{DH, FJ} over simple groups.

\bigskip 
 \subsection{Tensor square of the Steinberg character}
There is an interesting property of $St \otimes St$, the tensor square of the Steinberg character. 

 \begin{thm}\cite[Theorem 1.2]{heide2013conjugacy}
 Let $G$ be a finite simple group of Lie type, other than $PSU_n(q)$ with $n \geq 3$ coprime to $2(q+1)$. Then every irreducible character of $G$ is a constituent of the tensor square $St \otimes St$. 
 \end{thm}
 
 \subsubsection{Vanishing case}
 Let us consider this theorem in our case, i.e., \[
 \left\langle St \otimes St \otimes R_{T_w}^G(\theta),1 \right\rangle\]
 for $\theta \in \widehat{T_w^F}$. In this case, we have the following vanishing result. 
 
 \begin{cor}
 Let us assume that $Z(G)^F\neq \{1\}$. If $\theta$ is non-trivial on every non-trivial subgroup of $T_w^F$, then the multiplicity vanishes, i.e., $ \left\langle St \otimes St \otimes R_{T_w}^G(\theta),1 \right\rangle=0$.
 \end{cor}
 
 \begin{proof}
 The proof is identical to that of Corollary \ref{coro:12}.
 \end{proof}
This observation implies that  \cite[Theorem 1.2]{heide2013conjugacy}
does not generalise to arbitrary reductive groups. 
 Another reason is that the dimension of the Steinberg character of $\mathrm{GL}_2(\mathbb{F}_q)$ is $q$, so the dimension of its square tensor product is $q^2$. However, the sum of the dimension of every irreducible character is $q^3-q^2$, so the  tensor square of the Steinberg character cannot contain every irreducible character of $\mathrm{GL}_2(\mathbb{F}_q)$ when $q>2$.
 
\subsubsection{Non-vanishing case}
Now, let us consider $\theta \in \widehat{T_w^F}$, which is trivial at least on the centre of $G$. Then we have the following result on the multiplicity $\left\langle St \otimes St \otimes R_{T_w}^G(\theta), 1 \right\rangle $.

Let $(X,\Phi,X^\vee,\Phi^\vee)$ be the root datum of $G$, and the modulus of $G$, denoted by $d(G)$, is defined by the least common multiple of $|\mathrm{Tor}(X/\langle \Psi\rangle)|$, where $\Psi$ ranges over all closed subsystems of $\Phi$. For more details, please see \cite[\S4.3]{KNP}.
 
 \begin{cor}Let us assume that $q$ is sufficiently large, and $\theta \in \widehat{T_w^F}$  is in general position and trivial at least on the centre of $G^F$.
 When $G$ is not a product of groups of type $A_1$, the multiplicity $\left\langle St \otimes St \otimes R_{T_1}^G(\theta), 1 \right\rangle $ is non-zero. Furthermore, its degree is $ |\Phi^+|+\mathrm{dim}(Z(G))-\mathrm{rk}(G)$. In addition, when $w=1$ and $q\equiv1 \pmod{d(G)}$, the leading coefficient is the number of   connected components of $Z(G)$ (up to sign).
  \end{cor}
  The logic of the proof is the same as the proof of Corollary \ref{coro:degree-non-zero-condition}.
 \begin{proof} Note that $\left\langle St \otimes St \otimes \epsilon_G \epsilon_{T_w} R_{T_w}^G(\theta), 1 \right\rangle=\epsilon_G \epsilon_{T_w}\left\langle St \otimes St \otimes  R_{T_w}^G(\theta), 1 \right\rangle  $ is an integer since $St$ and $\epsilon_G \epsilon_{T_w}R_{T_w}^G(\theta)$ are actual characters.

Let us consider the degree of each term indexed by $[\fT,1]\in \Xi_w^{G^F}$.
For each $\fT$ and using Remark \ref{rem:split-}, the degree in $q$ of 
\begin{equation}\label{eq:coro14}\frac{|\fT^F|_pQ_{ T_w }^{\fT}(1)   }{|\fT^F|_{p'}|W_w(\fT)|}  \sum_{\underline{\fT}^F\in [\fT^F]} \sum_{v\in W_w(\underline{\fT}) }   \sum_{\substack{\fT'^F\in \mathcal{T}_w^{G^F}\\ \underline{\fT}^F \subset \fT'^F}} \mu_w(\underline{\fT}^F,\fT'^F)  \delta_{\dot{v}\cdot \theta,\fT'^F}
\end{equation} is at most  \[
 |\Phi(\fT)^+|+\mathrm{dim}(Z(\fT))-\mathrm{rk}(G).\]
 When $\fT=G$, then the degree of the corresponding term is exactly
 \[
 |\Phi^+|+\dim(Z(G))-\mathrm{rk}(G)
 \]
 since the corresponding M\"obius function has value $1$.
Then if we can show that the degree $ |\Phi^+|+\dim(Z(G))-\mathrm{rk}(G)$ is strictly larger than the other terms' upper bound ($ |\Phi(\fT)^+|+\mathrm{dim}(Z(\fT))-\mathrm{rk}(G)$), we can check that the degree of the multiplicity is $ |\Phi^+|+\dim(Z(G))-\mathrm{rk}(G)$. In other words, we need to show the following
\begin{equation}\label{eq:degree-computation}
 |\Phi^+|+\mathrm{dim}(Z(G))-( |\Phi(\fT)^+|+\mathrm{dim}(Z(\fT))) >0
\end{equation} for any proper pseudo-Levi subgroup $\fT$ of $G$.
Note that the inequality in \eqref{eq:degree-computation} holds for all types except $A_1$ with the following reasons.
If $\dim(Z(\fT))< \mathrm{rk}(G)$, then this is true from \cite[Lemma 6.1.3]{KNP}. 
So, let us assume that $\dim(Z(\fT))=\mathrm{rk}(G)$, which means that $\fT=T_1$, i.e., $|\Phi(\fT)^+|=0$. Then since $|\Phi^+|>\mathrm{rk}(G)-\dim(Z(G))=\dim(T_1/Z(G))$ if $G$ is not the product of $A_1$ types, so we can check that Equation \eqref{eq:degree-computation} holds.
 So the degree of $\left\langle St \otimes St \otimes R_{T_1}^G(\theta), 1 \right\rangle $ is $ |\Phi^+|+\mathrm{dim}(Z(G))-\mathrm{rk}(G).$
 
Now, let us consider $q$ as a power of prime, not a variable.
Under the condition $q\equiv1 \pmod{d(G)}$ and $w=1$, we can determine the leading coefficient uniformly. Note that the leading coefficient comes from the expression in Equation \eqref{eq:coro14} when $\fT=G$, i.e., $\frac{|G^F|_p  }{|G^F|_{p'} }  Q_{T_1}^{G}(1) |Z(G)^F|$. Then  the leading coefficient is $\epsilon_G\epsilon_{T_1}|\pi_0(Z(G))|$  from \cite[\S5.2]{KNP} (since we assume that our $G$ has simply connected derived subgroup) with the fact that $|G^F|_p$, $|G^F|_{p'}$ and $Q_{T_w}^{G}(1)$ are monic (up to sign).
 \end{proof}

 \section{Frobenius-Schur indicators of the modules over the Drinfeld	double }\label{s:3}

 For Hopf algebras, Frobenius-Schur indicators were studied from \cite{LM,FGSV}.  
The higher Frobenius–Schur indicators for semisimple Hopf algebras are related to some invariants of irreducible characters, for example,  the order, the multiplicity, the exponent, and the index. 
In particular, when we consider the group algebra $\mathbb{C}[H]$ for a finite group, we can get the corresponding higher Frobenius-Schur indicator is the same as the higher Frobenius-Schur indicator over $\widehat{H}$, cf. \cite[\S2.3]{KSZ}.
The Drinfeld double of a finite dimensional Hopf algebra $V$ is a Hopf algebra with underlying vector space $V^*\otimes V$. Kashina-Sommerh\"auser-Zhu considered   higher Frobenius-Schur indicators of a Drinfeld double of a Hopf algebra in \cite[\S6]{KSZ}, and Schauenburg considers   higher Frobenius-Schur indicators of a Drinfeld double of a finite group in \cite{sch} using conjugacy classes and character table of the finite group.

 The goal of this section is to consider an application of the previous computation to   higher Frobenius–Schur indicators of modules over a Drinfeld double of a finite group using results in \cite{sch}. 
Recall that a key ingredient of this indicator is the cardinality of the set
$G_m(g,z):= \{x\in G^F\,|\, x^m=(gx)^m=z\}$, cf. \cite[\S3]{sch}.
Let us define a map $$\gamma_m^z\,:\, G_z\rightarrow \mathbb{C}\quad \text{given by} \quad \gamma_m^z(g)=|G_m(g,z)|.$$ Then this is a class function on $G^F$ from \cite[Lemma and Definition 3.1]{sch}, and we have the following decomposition from \cite[Equation (4.5)]{sch}:
\[
\gamma_m^z=|G_z|\sum_{\chi \in \widehat{G_z^F}} \frac{|\left\langle \chi, w_m^z\right\rangle|^2}{\chi(1)}\chi,\quad \text{where }w_m^z(x)=\delta_{x^m,z}. 
\] So it is an interesting topic to  compute each coefficient $|\left\langle \chi , w_m^z\right\rangle |^2$ for each $\chi \in \widehat{G^F}$. 
	
\subsection{Semisimple elements}	As noted in \S\ref{ss:12}, let us consider the inner product $$\left\langle \chi, w_{\mathfrak{q}}^1\right\rangle$$
for $\mathfrak{q}:=|G^F|_{p'}$.  From the definition, we have $w_{\mathfrak{q}}^1(x)=\begin{cases}
		1 \quad \text{if } |x| \text{ divides } \mathfrak{q} \\
		0 \quad \text{otherwise},
		\end{cases}$
and so we need to compute
 \begin{equation}\label{eq:inner}
 \left\langle \chi, w_{\mathfrak{q}}^1\right\rangle=\frac{1}{|G^F|}\sum_{g\in G^F} \chi(g)w_{\mathfrak{q}}^1(g)= \frac{1}{|G^F|}\sum_{\substack{g\in G^F \\ g^{\mathfrak{q}}=1}} \chi(g) .
 \end{equation}
 
Recall that under the Jordan decomposition, $g=g_sg_u=g_ug_s$, the order of $g_u$ is   a power of $p$ and the order of $g_s$ is prime to $p$. This implies that if $g_u$ is non-identity, then $g^{\mathfrak{q}} \neq 1$. Therefore, we see that in Equation \eqref{eq:inner} it suffices to consider only semisimple elements of $G^F$. The computation uses an idea similar to the one employed in evaluating Equation \eqref{eq:mainstregu} in Theorem \ref{thm:main1intro}.
Let us recall the work of Deligne and Lusztig  for a semisimple element $s$ and $\chi \in \widehat{G^F}$ (\cite[Corollary 7.6]{DL76}): \[
 \chi(s)=\frac{1}{|G_s^F|_p}\sum_{(T,\theta)}\epsilon_{G_s}\epsilon_T\left\langle R_T^G(\theta),\chi\right\rangle \theta(s),
 \]
 where the sum runs over every $F$-stable maximal torus $T$ containing $s$ and for every $\theta\in \widehat{T^F}$.
 To compute $ \left\langle \chi, w_{\mathfrak{q}}^1\right\rangle$, let us consider the set $\mathbb{T}^{G^F}:= \{ [G_s^F]\,|\, s\in G^{ss,F}\}$ (recall that $G^{ss,F}$ is the set of semisimple elements in $G^F$), and the following map $\tau_{G^F}\,:\, G^{ss,F}\rightarrow \mathbb{T}^{G^F}$ given by $\tau_{G^F}(s)=[G_s^F]$.

Then our problem becomes
   \[\begin{split}
 \left\langle \chi, w_{\mathfrak{q}}^1\right\rangle & =  \frac{1}{|G^F|}\sum_{\substack{g\in G^F \\ g^{\mathfrak{q}}=1}}\frac{1}{|G_g^F|_p}\sum_{(T,\theta)}\epsilon_{G_g}\epsilon_T\left\langle R_T^G(\theta),\chi\right\rangle \theta(g) \\
 &=  \frac{1}{|G^F|}\sum_{[\fT^F]\in \mathbb{T}^{G^F}}\frac{1}{|\fT^F|_p} \sum_{s\in \tau_{G^F}^{-1}([\fT^F])}\sum_{(T,\theta)}\epsilon_{\fT}\epsilon_T\left\langle R_T^G(\theta),\chi\right\rangle \theta(s) \\
 &=  \frac{1}{|G^F|}\sum_{[\fT^F]\in \mathbb{T}^{G^F}}\frac{1}{|\fT^F|_p}  \sum_{(T,\theta)}\epsilon_{\fT}\epsilon_T\left\langle R_T^G(\theta),\chi\right\rangle   \sum_{s\in (\tau_{G^F}^{-1}([\fT^F])\cap T^F)/W(T) }\theta(s)  \\
 &=  \frac{1}{|G^F|}\sum_{[\fT^F]\in \mathbb{T}^{G^F}}\frac{1}{|\fT^F|_p}  \sum_{(T,\theta)}\epsilon_{\fT}\epsilon_T\left\langle R_T^G(\theta),\chi\right\rangle   \frac{1}{|(W(T)/W_\fT(T))^F|}\sum_{s\in  \tau_{G^F}^{-1}([\fT^F])\cap T^F }\theta(s)\\
 &=   \frac{1}{|G^F|} \sum_{[\fT^F ]\in \mathbb{T}^{G^F}}\frac{1}{|\fT^F|_p}  \sum_{(T,\theta)}\frac{\epsilon_{\fT}\epsilon_T\left\langle R_T^G(\theta),\chi\right\rangle }{ |(W(T)/W_\fT(T))^F|}  \sum_{\underline{\fT}^F\in [\fT^F]}\sum_{\substack{s\in T^F\text{ s.t.}\\G_s^F=\underline{\fT}^F} }\theta(s).
\end{split} \] 
Here, we consider $\theta(s)=0$ for $s\notin T^F$, but $\theta\in \widehat{T^F}$. 
Now, consider the set $\mathcal{T}_T^{G^F}:=\{ G_s^F\,|\, s\in T^F\}$ for a maximal torus $T$ and the M\"obius function $\mu_T:= \mu_{\mathcal{T}_T^{G^F}}$ on  $\mathcal{T}_T^{G^F}$, where the partial ordering is the inclusion. 
 Then, from our work in \S\ref{s:2}, we can easily conclude that \[\sum_{\substack{s\in T^F\text{ s.t.}\\G_s^F=\fT^F} }\theta(s)=\sum_{\substack{\fT'^F\in \mathbb{T}_T^{G^F}\\ \fT^F\subset \fT'^F}} \mu_T(\fT,\fT')\sum_{\substack{s\in T^F \\ \fT'^F\subset G_s^F} }\theta(s)=\sum_{\substack{\fT'^F\in  \mathbb{T}_T^{G^F}\\ \fT^F\subset \fT'^F}} \mu_T(\fT^F,\fT'^F)\delta_{\theta,\fT'^F}.
 \] 
 Note that if $G_s^F\neq\fT^F$ for any $s\in T^F$, 
 then the left-hand side sum is zero. So for ease of notation, in this case, we consider $\mu_T(\fT^F,\fT'^F)$ as zero.
Then we obtain the following result:
 \begin{thm}
 We have 
 \[
 \left\langle \chi, w_{\mathfrak{q}}^1\right\rangle =    \sum_{[\fT^F ]\in \mathbb{T}^{G^F}}\frac{1}{|G^F| |\fT^F|_p }  \sum_{(T,\theta)}\frac{\epsilon_{\fT}\epsilon_T\left\langle R_T^G(\theta),\chi\right\rangle   }{|(W(T) /W_\fT(T))^F|}\sum_{\underline{\fT}^F\in [\fT^F]}\sum_{\substack{\fT'^F\in  \mathcal{T}_T^{G^F}\\ \underline{\fT}^F\subset \fT'^F}} \mu_T(\underline{\fT}^F,\fT'^F)\delta_{\theta,\fT'^F}.
 \]
 \end{thm}

\begin{rem}
When we consider the value $\gamma_{\mathfrak{q}}^1(1)=|G_{\mathfrak{q}}(1,1)|=|\{x\in G^F\,|\, x^{\mathfrak{q}}=1\}|$, this is the number of semisimple elements in $G^F$. 
 More explicitly, we have
 \[\begin{split}
 \gamma_{\mathfrak{q}}^1(1)&=|G^F|\sum_{\chi \in \widehat{G^F}} |\left\langle \chi , w_\mathfrak{q}^1\right\rangle|^2\\
 &=\frac{1}{|G^F|}\sum_{\chi \in \widehat{G^F}}  \left| \sum_{[\fT^F]\in \mathbb{T}^{G^F}}\frac{1}{  |\fT^F|_p}  \sum_{(T,\theta)}  \frac{ \epsilon_{\fT}\epsilon_T\left\langle R_T^G(\theta),\chi\right\rangle }{|(W(T)/W_\fT(T))^F|}\sum_{\underline{\fT}^F\in [\fT^F]}\sum_{\substack{\fT'^F\in  \mathcal{T}_T^{G^F}\\ \underline{\fT}^F\subset \fT'^F}} \mu_T(\underline{\fT}^F,\fT'^F)\delta_{\theta,\fT'^F}\right|^2.
 \end{split}
 \]
 \end{rem}

  \begin{rem}
 Another interesting point to discuss is the value of
$\left\langle \chi, w_{\tilde{\mathfrak{q}}} \right\rangle$
for $\tilde{\mathfrak{q}} = |G^F|_p$.
Using the same reasoning as for $\mathfrak{q}$, we only need to consider unipotent elements in order to compute
$\left\langle \chi, w_{\tilde{\mathfrak{q}}}^1 \right\rangle$; indeed,
$\left\langle \chi, w_{\tilde{\mathfrak{q}}}^1 \right\rangle
= \frac{1}{|G^F|}
\sum_{\substack{g \in G^{\mathrm{uni},F}}}
\chi(g)$.
In general, there is no uniform method to compute $\chi(g)$ for every unipotent element $g$.
However, when we fix \(G = \GL_n\), it is known that every irreducible character of \(\GL_n^F\) can be expressed as a linear combination of Deligne–Lusztig characters; see \cite[\S3]{LS}.
Therefore, in this case, $\left\langle \chi, w_{\tilde{\mathfrak{q}}}^1 \right\rangle$ can be computed for each \(\chi \in \widehat{\GL_n^F}\) as a sum of values of Green functions.
We omit the explicit computation here.
  \end{rem}

 \bigskip

 \section{Multiplicity of unipotent almost   characters}\label{s:4}
When we study irreducible characters of finite reductive groups, unipotent characters play a crucial role. So when we  consider multiplicity, it is a natural question to compute the multiplicity of tensor product of unipotent characters. This problem is considered by Hiss and L\"ubeck for small rank cases in \cite{HL} (also with unipotent almost characters), and by Letellier for finite general linear groups in \cite{letellier2013tensor}. Note that every unipotent almost character of $\GL_n(\Fq)$ is a unipotent   character, and vice versa. 

While studying the results in \S\ref{s:2}, we realised that we could also consider the multiplicity of the tensor product of unipotent almost characters following the idea of \cite{HL,letellier2013tensor}. Consequently, in this section, we compute this multiplicity 
$$\left\langle U_{\chi_1}\otimes U_{\chi_2}\otimes\cdots \otimes  U_{\chi_m},1 \right\rangle=\frac{1}{|G^F|}\sum_{g\in G^F} U_{\chi_1}(g) U_{\chi_2}(g) \cdots U_{\chi_m}(g).$$    This serves as a generalisation of \cite{HL}. A natural subsequent question is how to determine when this multiplicity vanishes or is non-zero.

To compute the multiplicity, recall the following formula for the value of the Deligne-Lusztig character from Proposition \ref{thm:charactervalue} and Equation \eqref{eq:DLcharacter-with-semisimple}: \[
R_{T_w}^G(1)(g)=\frac{1}{|G_{g_s}^{F}|}\sum_{\substack{x\in G^F\text{ s.t.}\\ x^{-1}g_sx\in T_w^F}}Q_{xT_wx^{-1}}^{G_{g_s}}(g_u)=\sum_{v\in W_w(\fT)}Q_{\tilde{v}^{-1}T_w\tilde{v}}^{G_{g_s}}(g_u).
\]

 \subsection{Multiplicity}
Let us consider the following decomposition over $\Xi^{G^F}$ as introduced in \S\ref{ss:types}:
 \[\begin{split}\left\langle U_{\chi_1}\otimes U_{\chi_2}\otimes\cdots \otimes  U_{\chi_m},1 \right\rangle&=\frac{1}{|G^F|}\sum_{g\in G^F} U_{\chi_1}(g) U_{\chi_2}(g)\cdots U_{\chi_m}(g)\\
  &=\frac{1}{|G^F|}\sum_{\xi \in \Xi^{G^F}}\sum_{g\in \omega_{G^F}^{-1}(\xi)} U_{\chi_1}(g) U_{\chi_2}(g)\cdots  U_{\chi_m}(g).\end{split}\]
The problem is   reduced to computing the term $\sum_{g\in \omega_{G^F}^{-1}(\xi)} U_{\chi_1}(g) U_{\chi_2}(g) \cdots U_{\chi_m}(g)$ for each $\xi \in \Xi^{G^F}$.

\subsubsection{} Let 
$\xi=[\fT^F, u]$. Then we have that
\begin{equation*}\label{eq:partialsumtau}
\begin{split}
\sum_{g\in \omega_{G^F}^{-1}(\xi)} U_{\chi_1}(g) U_{\chi_2}(g) \cdots U_{\chi_m}(g)&=\frac{1}{|W(T_1)|^m}\sum_{g\in \omega_{G^F}^{-1}(\xi)}\prod_{i=1}^m\left (\sum_{\substack{w \in W(T_1) }} \chi_i(w)R_{T_{w}}^G(1)(g)\right)\\
&=\frac{1}{|W(T_1)|^m}\sum_{g\in \omega_{G^F}^{-1}(\xi)} \prod_{i=1}^m\left(\sum_{w\in W(T_1)} \chi_i(w)\sum_{v\in W_{w}(\fT)}Q_{{\tilde{v}^{-1}T_{w }\tilde{v} }}^{\fT}(g_u)\right)\\
&=\frac{|\omega_{G^F}^{-1}(\xi)|}{|W(T_1)|^m} \prod_{i=1}^m\left(\sum_{w\in W(T_1)} \chi_i(w)\sum_{v\in W_{w}(\fT)}Q_{{\tilde{v}^{-1}T_{w }\tilde{v} }}^{\fT}(u)\right)
\end{split}
\end{equation*}
where the last equality comes from the property that the Green function $Q_{T_{w }}^{\fT}$ is a class function over $\fT^F$. Recall that we ignore $R_{T_{w }}^G(1)(g)$ in the sum when $g_s$ is not $G^F$-conjugate to an element of $T_{w }^F$ with the same reason in \S\ref{intro:multiplicity}.
Using this observation, we obtain the following result.
\begin{thm}We have
\[\begin{split}
&\left\langle U_{\chi_1}\otimes U_{\chi_2}\otimes \cdots \otimes U_{\chi_m},1 \right\rangle\
=\frac{1}{|G^F||W(T_1)|^m} \sum_{\xi=[\fT^F,u] \in \Xi^{G^F}}  |\omega_{G^F}^{-1}(\xi)|\prod_{i=1}^m\left(\sum_{w\in W(T_1)} \chi_i(w)\sum_{v\in W_{w}(\fT)}Q_{{\tilde{v}^{-1}T_{w }\tilde{v}}}^{\fT}(u)\right),
\end{split}
\]
with considering $Q_{\tilde{v}^{-1}T_{w}\tilde{v}}^\fT=0$ when $\tilde{v}^{-1}T_{w}\tilde{v}$ is not in $\fT$ up to $G^F$-conjugation. 
\end{thm}

\subsection{Contribution of irreducible characters of the Weyl group}\label{ss:43}
 
Our future goal is to find when the multiplicity
$ \langle U_{\chi_1} \otimes U_{\chi_2} \otimes \cdots \otimes U_{\chi_m},\, 1  \rangle$
vanishes or not.
For this purpose, we believe that we need to compute term-by-term over types carefully. For this reason, we focus on special terms
 $\xi = [T_w, 1]$. Here, for any $g \in \omega_{G^F}^{-1}(\xi)$, the element $g$ is regular semisimple and lies in $T_w$ (up to conjugation).
  \begin{lem}
For a type $\xi=[T_w^F,1]\in \Xi^{G^F}$ and the conjugacy class $[w]$ of $w$ in $W(T_1)$, we have 
\[
\sum_{g\in \omega_{G^F}^{-1}(\xi)} U_{\chi_1}(g) U_{\chi_2}(g) \cdots U_{\chi_m}(g)=
|\omega_{G^F}^{-1}(\xi)| \chi_1(w)\cdots \chi_m(w).
\]
\end{lem}
\begin{proof}
Let us consider a regular element $s$ in $T_w^F$ up to $G^F$-conjugation. Then $s$ is not in any other $T_v$ up to $G^F$-conjugation if $w$ and $v$ are not $W(T_1)$-conjugate. This implies that  we only need to consider those $w' \in [w]$ from Theorem \ref{eq:prop6}. Hence with $W_w(T_w)=W(T_w)^F$, we obtain
\[\begin{split}
\sum_{g\in \omega_{G^F}^{-1}(\xi)} U_{\chi_1}(g) U_{\chi_2}(g) \cdots U_{\chi_m}(g)&=\frac{|\omega_{G^F}^{-1}(\xi)|}{|W(T_1)|^m}\prod_{i=1}^m\left(\sum_{{\bar{w}}\in W(T_1)} \chi_i({\bar{w}})\sum_{v\in W_{{\bar{w}}}(T_w)}Q_{{\tilde{v}^{-1}T_{{\bar{w}} }\tilde{v}}}^{T_w}(u)\right)\\
&= \frac{|\omega_{G^F}^{-1}(\xi)|}{|W(T_1)|^m}\prod_{i=1}^m \left( \sum_{\substack{\underline{w} \in [w] }}\chi_i(\underline{w} ){|W(T_{\underline{w} })^F|}\right)
\\
&= \frac{|\omega_{G^F}^{-1}(\xi)|}{|W(T_1)|^m}{|W(T_{w})^F|^m}\prod_{i=1}^m \left( \sum_{\underline{w}\in [w]} \chi_i(\underline{w})\right)
\\
&= \frac{|\omega_{G^F}^{-1}(\xi)|}{|W(T_1)|^m}{|W(T_{w})^F|^m}|[w]|^m\prod_{i=1}^m \chi_i(w).
\end{split}
\]
Since $W(T_w)^F\simeq C_W(w) $ from \cite[Proposition 3.3.6]{carter1985finite}, we have $|W(T_{w})^F||[w]|=|W(T_1)|$, which completes the proof.
\end{proof}

\begin{lem}\label{lem:coefficientofinverseregular} For any $\xi=[T_w^F,1]$, we have  $|\omega_{G^F}^{-1}(\xi)| =\frac{|[w]|}{|W(T_1)|}q^{\mathrm{dim} (G)}+(\text{lower degree terms})$.
\end{lem}
\begin{proof}
The coefficient can be computed using 
$N_G(T_w)^F/C_G(T_w)^F = N_G(T_w)^F/T_w^F \simeq C_W(w)$ 
from \cite[Proposition 3.3.6]{carter1985finite}, together with the fact that 
$C_{\,N_G(T_w)^F/T_w^F}(t) = \{1\}$ 
for a regular element $t \in T_w$.  The degree of $\omega_{G^F}^{-1}(\xi)$ comes from the facts that $\dim(\omega_{G }^{-1}(\xi)\cap T_w )=\mathrm{rk}(G)$ (cf. \cite[Lemma 2.3.11]{geck2020character}) and $\dim(G/G_s)=\dim(G)-\mathrm{rk}(G)$ for any $s\in \omega_{G^F}^{-1}(\xi)$.\end{proof}
 

\subsubsection{}This observation implies the following:
  \begin{equation*}
  \begin{split}&
  \left\langle U_{\chi_1}\otimes U_{\chi_2}\otimes \cdots \otimes U_{\chi_m},1 \right\rangle\\
  &=\frac{1}{|G^F|}\sum_{\xi \in \Xi^{G^F}}\sum_{g\in \omega_{G^F}^{-1}(\xi)}
  U_{\chi_1}(g) U_{\chi_2}(g)\cdots  U_{\chi_m}(g) \\
  &=\frac{1}{|G^F|}\sum_{\substack{\xi=[\fT^F,1] \in\Xi^{G^F}\\ \fT\ \text{is a torus}}}\sum_{g\in \omega_{G^F}^{-1}(\xi)}
  U_{\chi_1}(g) U_{\chi_2}(g) \cdots U_{\chi_m}(g)+\frac{1}{|G^F|}\sum_{\substack{\xi=[\fT^F,u] \in \Xi^{G^F}\\\fT\ \text{is not a torus}}}\sum_{g\in \omega_{G^F}^{-1}(\xi)}
  U_{\chi_1}(g) U_{\chi_2}(g) \cdots U_{\chi_m}(g) \\
  &=\frac{1}{|G^F|}\sum_{\substack{\xi=[\fT^F,1] \in \Xi^{G^F}\\ \fT=T_w\text{ for some }w}}|\omega_{G^F}^{-1}(\xi)| \chi_1(w)\cdots \chi_m(w)+\frac{1}{|G^F|}\sum_{\substack{\xi=[\fT^F,u] \in \Xi^{G^F}\\ \fT\ \text{is not a torus}}}\sum_{g\in \omega_{G^F}^{-1}(\xi)}
  U_{\chi_1}(g) U_{\chi_2}(g)\cdots  U_{\chi_m}(g)
 \end{split}\end{equation*}
 Let $[W(T_1)]$ be the conjugacy classes of $W(T_1)$, and then we can give the following result.
 \begin{lem}We have
  \begin{equation}\label{eq:toruspartleadingterm}
  \begin{split}&
  |G^F|\left\langle U_{\chi_1}\otimes U_{\chi_2}\otimes\cdots \otimes U_{\chi_m},1 \right\rangle  \\
  &=\sum_{[w]\in [W(T_1)]} \frac{|[w]|}{|W(T_1)|}\chi_1(w)\chi_2(w)\cdots \chi_m(w) q^{\mathrm{dim} (G)}+(\text{lower degree terms})\\
  &\quad + \sum_{\substack{\xi=[\fT^F,u] \in \Xi^{G^F}\\ \fT\ \text{is not a torus}}}\sum_{g\in \omega_{G^F}^{-1}(\xi)}
  U_{\chi_1}(g) U_{\chi_2}(g) \cdots U_{\chi_m}(g) \\
  &=\left\langle \chi_1\otimes \chi_2\otimes\cdots \otimes \chi_m,1\right\rangle_{W(T_1)} \cdot  q^{\mathrm{dim} (G)}+(\text{lower degree terms})\\
  & \quad+ \sum_{\substack{\xi=[\fT^F,u] \in \Xi^{G^F}\\ \fT\ \text{is not a torus}}}\sum_{g\in \omega_{G^F}^{-1}(\xi)}
  U_{\chi_1}(g) U_{\chi_2}(g) \cdots U_{\chi_m}(g).
 \end{split}\end{equation}
 \end{lem}

\subsubsection{Beyond general linear group}
Now let us recall a result of Letellier. He showed that in the case of \(\GL_n^F\), if
$
\left\langle \chi_1 \otimes \chi_2 \otimes \cdots \otimes \chi_m, 1 \right\rangle_{S_n} \neq 0,
$
then
$
\left\langle U_{\chi_1} \otimes U_{\chi_2} \otimes \cdots \otimes U_{\chi_m}, 1 \right\rangle_{\GL_n^F} \neq 0,
$
cf.~\cite{letellier2013tensor} or \cite[Theorem 2.3.1]{letellier2023saxl}. Then we can anticipate the same phenomenon from Equation \eqref{eq:toruspartleadingterm}. However,   this is not true in some other cases. When we consider $G=\mathrm{Sp}_4$, then the multiplicity $\langle \zeta_1\otimes \zeta_2 \otimes \zeta_3,1\rangle_{W(T_1)}=1$, where $\zeta_1,\zeta_2$ and $\zeta_3$ are all distinct one-dimensional  non-trivial characters of $W(T_1)$. However, from \cite[$C_{2,sc}(q)$]{Mat}, we can check that the corresponding  multiplicity of unipotent almost characters $1,2,5$ (following their notation) vanishes. Therefore, the statement `$
\left\langle \chi_1 \otimes \chi_2 \otimes \cdots \otimes \chi_m, 1 \right\rangle_{W(T_1)} \neq 0,
$
then
$
\left\langle U_{\chi_1} \otimes U_{\chi_2} \otimes \cdots \otimes U_{\chi_m}, 1 \right\rangle_{G^F} \neq 0
$' does not hold in general. Therefore, finding a condition which guarantees that the non-zero value of the multiplicity $
\left\langle U_{\chi_1} \otimes U_{\chi_2} \otimes \cdots \otimes U_{\chi_m}, 1 \right\rangle_{G^F} \neq 0
$ would be an interesting question.

 \bigskip
\noindent \textbf{Acknowledgments}		
GyeongHyeon Nam was supported by Oscar Kivinen's Väisälä project grant of the Finnish Academy of Science and Letters.
The author is very grateful to Jungin Lee and Anna Pusk\'as for helpful advice.


\begin{bibdiv}
\begin{biblist}

 
		
		
		
\bib{carter1985finite}{book}{
    AUTHOR = {Carter, R. W.},
     TITLE = {Finite groups of {L}ie type},
    SERIES = {Wiley Classics Library},
      NOTE = {Conjugacy classes and complex characters,
              Reprint of the 1985 original,
              A Wiley-Interscience Publication},
 PUBLISHER = {John Wiley \& Sons, Ltd., Chichester},
      YEAR = {1993},
     PAGES = {xii+544},
}




\bib{DL76}{article}{
    AUTHOR = {Deligne, P.},
    author={Lusztig, G.},
     TITLE = {Representations of reductive groups over finite fields},
   JOURNAL = {Annals of Mathematics  },
     VOLUME = {103},
      YEAR = {1976},
    NUMBER = {1},
     PAGES = {103--161},
      ISSN = {0003-486X},
}

\bib{DH}{article}{
  title={The M{\"o}bius function of the lattice of closed subsystems of a root system},
  author={Deriziotis, D. I.},
   author={ Holt, D. F.},
  journal={Communications in Algebra},
  volume={21},
  number={5},
  pages={1543--1570},
  year={1993},
  publisher={Taylor \& Francis}
}




\bib{eti}{book}{
  title={Introduction to representation theory},
  author={Etingof, P. I.},
   author={ Golberg, O.},
    author={ Hensel, S.},
     author={ Liu, T.},
      author={ Schwendner, A.},
       author={ Vaintrob, D.},
        author={ Yudovina, E.},
  volume={59},
  year={2011},
  publisher={American Mathematical Soc.}
}
 
\bib{FJ}{article}{
  title={The lattices and M{\"o}bius functions of stable closed subrootsystems and hyperplane complements for classical Weyl groups},
  author={Fleischmann, P.},
   author={ Janiszczak, I.},
  journal={manuscripta mathematica},
  volume={72},
  number={1},
  pages={375--403},
  year={1991},
  publisher={Springer}
}


\bib{FGSV}{article}{
    AUTHOR = {Fuchs, J.},
     author={ Ganchev, A. Ch.}, 
     author={ Szlach\'anyi, K.},
      author={Vecserny\'es, P.},
     TITLE = {{$S_4$} symmetry of {$6j$} symbols and {F}robenius-{S}chur
              indicators in rigid monoidal {$C^*$} categories},
JOURNAL = {Journal of Mathematical Physics},
    VOLUME = {40},
      YEAR = {1999},
    NUMBER = {1},
     PAGES = {408--426},
}


\bib{geck2025}{article}{
    AUTHOR = {Geck, M.},
     TITLE = {On the computation of character values for finite {C}hevalley
              groups of exceptional type},
   JOURNAL  = {Pure and Applied Mathematics Quarterly},
    VOLUME = {21},
      YEAR = {2025},
    NUMBER = {1},
     PAGES = {275--325},
}

\bib{geck2020character}{book}{
  title={The character theory of finite groups of Lie type: a guided tour},
  author={Geck, M.},
   author={ Malle, G.},
  volume={187},
  year={2020},
  publisher={Cambridge University Press}
}



\bib{HLRV}{article}{
    AUTHOR = {Hausel, T.},
    author={ Letellier, E.},
    author={
              Rodriguez-Villegas, F.},
     TITLE = {Arithmetic harmonic analysis on character and quiver
              varieties},
  JOURNAL = {Duke Mathematical Journal},
    VOLUME = {160},
      YEAR = {2011},
    NUMBER = {2},
     PAGES = {323--400}
}

\bib{heide2013conjugacy}{article}{
  title={Conjugacy action, induced representations and the Steinberg square for simple groups of Lie type},
  author={Heide, G.},
   author={ Saxl, J.},
  author={ Tiep, P. H.},
   author={ Zalesski, A. E.},
  journal={Proceedings of the London Mathematical Society},
  volume={106},
  number={4},
  pages={908--930},
  year={2013},
  publisher={Wiley Online Library}
}

\bib{HL}{incollection}{
    AUTHOR = {Hiss, G.},
    AUTHOR={ L\"ubeck, F.},
     TITLE = {Some observations on products of characters of finite
              classical groups},
 BOOKTITLE = {Finite groups 2003},
     PAGES = {195--207},
 PUBLISHER = {Walter de Gruyter, Berlin},
      YEAR = {2004},
      ISBN = {3-11-017447-2},
   MRCLASS = {20C33 (20G40)},
  MRNUMBER = {2125073},
MRREVIEWER = {Toshiaki\ Shoji},
}
 
 \bib{HL2}{article}{
  title={Induced cuspidal representations and generalised Hecke rings},
  author={Howlett, R. B.},
   author={ Lehrer, G. I.},
  journal={Inventiones mathematicae},
  volume={58},
  number={1},
  pages={37--64},
  year={1980},
  publisher={Springer}
}

\bib{IMM}{article}{
    AUTHOR = {Iovanov, M.},
     author={ Mason, G.},
     author={ Montgomery, S.},
     TITLE = {{$FSZ$}-groups and {F}robenius-{S}chur indicators of quantum
              doubles},
  JOURNAL = {Mathematical Research Letters},
    VOLUME = {21},
      YEAR = {2014},
    NUMBER = {4},
     PAGES = {757--779},
}



\bib{KNP}{article}{
	title={Arithmetic geometry of character varieties with regular monodromy},
  author={Kamgarpour, M.},
author={ Nam, G.},
author={ Pusk{\'a}s, A.},
	journal={Representation Theory},
	volume={29},
	number={11},
	pages={347--378},
	year={2025}
}



\bib{KNWG}{article}{
  title={Counting points on generic character varieties},
  author={Kamgarpour, M.},
   author={ Nam, G.},
    author={ Whitbread, B.},
     author={ Giannini, S.},
     JOURNAL = {Mathematical Research Letters},
    VOLUME = {33},
      YEAR = {2026},
    NUMBER = {2},
     PAGES = {341--395}
}


\bib{KSZ}{article}{
  title={On higher Frobenius-Schur indicators},
  author={Kashina, Y.},
  author={ Sommerh{\"a}user, Y.},
   author={ Zhu, Y.},
  volume={181},
  number={855},
   PAGES = {viii+65},
  year={2006},
JOURNAL = {Memoirs of the American Mathematical Society},
}

 \bib{MOflow}{misc}{
    TITLE = {Unipotent almost characters},
      author={Lafes},
    HOWPUBLISHED = {MathOverflow},
        JOURNAL = {MathOverflow},
    NOTE = {MathOverflow \href{https://mathoverflow.net/q/500406}{https://mathoverflow.net/q/500406}},
    EPRINT = {https://mathoverflow.net/q/500406},
    URL = {https://mathoverflow.net/q/500406}
}

\bib{lehrer1996space}{article}{
  title={The space of invariant functions on a finite Lie algebra},
  author={Lehrer, G.},
  journal={Transactions of the American Mathematical Society},
  volume={348},
  number={1},
  pages={31--50},
  year={1996}
}

 


\bib{letellier2013tensor}{article}{
  title={Tensor products of unipotent characters of general linear groups over finite fields},
  author={Letellier, E.},
  journal={Transformation Groups},
  volume={18},
  number={1},
  pages={233--262},
  year={2013},
  publisher={Springer}
}

 
\bib{letellier2023saxl}{article}{
  title={The Saxl conjecture and the tensor square of unipotent characters of $\GL_n (q) $},
  author={Letellier, E.},
   author={ Nam, G.},
  journal={Algebraic Combinatorics},
  volume={8},
  number={4},
  pages={1119--1140},
  year={2025}
}



 


 

\bib{LM}{article}{
  title={A Frobenius--Schur theorem for Hopf algebras},
  author={Linchenko, V.},
   author={ Montgomery, S.},
  journal={Algebras and Representation Theory},
  volume={3},
  number={4},
  pages={347--355},
  year={2000},
  publisher={Springer}
}

 \bib{Mat}{misc}{
title={\href{https://www.math.rwth-aachen.de/~Frank.Luebeck/preprints/tprdexmpl/tendec.html}{https://www.math.rwth-aachen.de/~Frank.Luebeck/preprints/tprdexmpl/tendec.html}},
author={ L\"ubeck, F.}}



\bib{lusztig}{book}{
  title={Characters of reductive groups over a finite field},
  author={Lusztig, G.},
 SERIES = {Annals of Mathematics Studies},
    VOLUME = {107},
 PUBLISHER = {Princeton University Press, Princeton, NJ},
      YEAR = {1984},
     PAGES = {xxi+384},
      ISBN = {0-691-08350-9; 0-691-08351-7},
   MRCLASS = {20G05 (14L20 20C15)},
  MRNUMBER = {742472},
MRREVIEWER = {Bhama\ Srinivasan},
       DOI = {10.1515/9781400881772},
       URL = {https://doi.org/10.1515/9781400881772},
}


\bib{LS}{article}{
  title={The characters of the finite unitary groups},
  author={Lusztig, G.},
   author={ Srinivasan, B.},
  journal={Journal of Algebra},
  volume={49},
  number={1},
  pages={167--171},
  year={1977},
  publisher={Academic Press}
}



 
 \bib{Nam}{article}{
  title={Absolutely indecomposable quasi-parabolic $ G $-bundles and the multiplicity of irreducible characters},
  author={Nam, G.},
  journal={arXiv preprint arXiv:2605.26963},
  year={2026}
}

 \bib{sch}{article}{
 	title={Higher Frobenius-Schur indicators for Drinfeld doubles of finite groups through characters of centralizers},
 	author={Schauenburg, P.},
 	journal={arXiv preprint arXiv:1604.02378},
 	year={2016}
 } 
 
 
\bib{springer1980steinberg}{article}{
  title={The Steinberg function of a finite Lie algebra},
  author={Springer, T. A.},
  journal={Inventiones mathematicae},
  volume={58},
  number={3},
  pages={211--215},
  year={1980},
  publisher={Springer}
}

 
 



\end{biblist}
\end{bibdiv}
\end{document}